\documentclass{article}

\usepackage{amsmath,amsthm,amssymb,mathtools}  
\usepackage{dsfont}			
\usepackage{thmtools}		
\usepackage{enumerate}
\usepackage{graphicx}		
\usepackage{subcaption}		
\usepackage[font=small]{caption}	

\usepackage[T1]{fontenc}	

\hypersetup{colorlinks=true, linkcolor=blue, citecolor=blue}

\usepackage[nameinlink]{cleveref} 	

\usepackage{bm}					
\usepackage{cmll}				

\usepackage{makecell}			
\usepackage{booktabs}			
\renewcommand{\arraystretch}{2}	

\usepackage[latin1]{inputenc}                       

\usepackage{verbatim}		
	
\usepackage[disable,		
	colorinlistoftodos,textsize=scriptsize,textwidth=3.7cm]{todonotes}		
\newcommand{\SELF}[1]{\todo[color=green!40]{#1}} 
\newcommand{\OMIT}[1]{\todo[color=gray!30]{#1}}  
\newcommand{\CITE}[1]{\todo[color=cyan!30]{#1}}  

\newtheorem{theorem}{Theorem}[section]
\newtheorem{proposition}[theorem]{Proposition}
\newtheorem{corollary}[theorem]{Corollary}
\newtheorem{lemma}[theorem]{Lemma}

\theoremstyle{definition}
\newtheorem{definition}[theorem]{Definition} 
\newtheorem{example}[theorem]{Example}

\declaretheoremstyle[
spaceabove=6pt, spacebelow=6pt,
headfont=\normalfont\bfseries,
notefont=\normalfont\bfseries, 
notebraces={}{},
bodyfont=\normalfont\itshape
]{Estilo1}

\newcommand\inner[1] 		{\langle #1 \rangle}

\def\N{\mathds{N}}
\def\Z{\mathds{Z}}
\def\R{\mathds{R}}
\def\C{\mathds{C}}
\def\F{\mathds{F}}
\def\PR{\mathds{P}}

\def\I{\mathcal{I}}

\def\WW{\mathcal{W}}
\def\LL{\mathcal{L}}
\def\II{\mathcal{I}}		
\def\CC{\mathcal{C}}
\def\AA{\mathbf{A}}
\def\BB{\mathbf{B}}
\def\CC{\mathbf{C}}
\def\PP{\mathbf{P}}
\def\ii{\mathbf{i}}				
\def\jj{\mathbf{j}}				

\def\im	 			{\mathrm{i}}

\def\Plucker		{Pl\"{u}cker}

\def\pperp		{\simperp}

\DeclareMathOperator{\Span}{span}

\DeclareMathOperator{\Ann}{Ann}

\DeclareMathOperator{\Tr}{Tr}
\DeclareMathOperator{\Hom}{Hom}
\DeclareMathOperator{\diam}{diam}
\DeclareMathOperator{\area}{area}

\def\lcontr		{\lrcorner\,}				

\def\wrt					{w.r.t.\ }
\def\ie						{i.e.\ }
\def\eg						{e.g.\ }

\begin{document}

\title{Asymmetric Metrics on the Full Grassmannian of Subspaces of Different Dimensions}

\author{Andr\'e L. G. Mandolesi 
               \thanks{Instituto de Matemática e Estatística, Universidade Federal da Bahia, Av. Adhemar de Barros s/n, 40170-110, Salvador - BA, Brazil. E-mail: \texttt{andre.mandolesi@ufba.br}}}
               
\date{\today}

\maketitle

\abstract{
Metrics on Grassmannians have a wide array of applications:
machine learning, wireless communication, computer vision, etc.
But the available distances between subspaces of distinct dimensions present problems,
and the dimensional asymmetry of the subspaces calls for the use of asymmetric metrics.
We extend the Fubini-Study metric as an asymmetric angle with useful properties,
and whose relations to products of Grassmann and Clifford geometric algebras make it easy to compute. 
We also describe related angles that provide extra information, and a method to extend other Grassmannian metrics to asymmetric metrics on the full Grassmannian.

\vspace{.5em}

\vspace{3pt}

}

\section{Introduction}

Subspaces represent data in many areas:
machine learning \cite{Hamm2008,Huang2018,Lerman2011}, 
	\CITE{Zhang2018}
computer vision \cite{Lui2012,Turaga_2008,Vishwanathan2006},
	\CITE{wireless communication: Du2018, Love2003, Zheng2002}
coding theory \cite{Barg2002},  
etc.
	\CITE{Ashikhmin2010 \\ language processing Hall2000 \\
	recommender systems Boumal2015}
Data sets are compared using various metrics on Grassmannians, sets of subspaces of a given dimension \cite{Edelman1999,Qiu2005,Stewart1990,Ye2016}: Fubini-Study, geodesic, chordal, etc. 
The first one is natural for quantum computation and other uses of quantum theory \cite{Bengtsson2017,Nielsen2010,Ortega2002},
	\CITE{Só Fubini-Study em $\PR(H)$. Also	Brody2001, Yu2019}
appearing also in wireless communication 
\cite{Agrawal2001,Dhillon2008,Love2005}.
Grassmannians of complex subspaces are used in both areas.
	\CITE{Agrawal2001, Dhillon2008, Love2003, Love2005, Pereira2021}
	
The full Grassmannian of subspaces of different dimensions is used for
image recognition \cite{Basri2011,Draper2014,Sun_2007,Wang_2006},
numerical linear algebra \cite{Beattie2005,Sorensen2002},
wireless communication \cite{Pereira2022},
	\CITE{Pereira2021}
etc.
But distances available for subspaces of distinct dimensions, like the 
gap and containment gap \cite{Beattie2005,Kato1995,Sorensen2002}, 
directional and symmetric distances \cite{Sun_2007,Wang_2006}, 
	\CITE{Bagherinia2011, Figueiredo2010, Sharafuddin2010, Zuccon2009}
projection Frobenius \cite{Basri2011,Draper2014,Pereira2022},
	\CITE{Pereira2021}
and others \cite{Ye2016}, have shortcomings: some are not metrics, others give little information or lack useful properties.

Part of the problem is a natural asymmetry between subspaces of distinct dimensions, which usual (symmetric) metrics fail to express.
Intuitively, a small region around a line does not contain a plane, but a neighborhood of a plane contains lines;
a line can be closer or farther from being in a given plane, but a plane is never any closer to being contained in a line;
a plane can be close to containing a given line, but not vice versa.
Formalizing these ideas requires an asymmetric metric.

Asymmetric metrics \cite{Anguelov2016,Kunzi2001,Mennucci2014}, in which the distance, time, cost, etc. to go from $A$ to $B$ is not the same as from $B$ to $A$, appear naturally in many situations: \eg a city with one-way streets, rush-hour traffic, going uphill or downhill, etc.
An example on the full Grassmannian is the containment gap,
	\CITE{Beattie2005,Sorensen2002}
but its asymmetry is usually neglected instead of put to good use, and many authors use instead the (symmetrized) gap.
	\CITE{Kato1995}
Both gaps are rough distances, not suited for all applications.

The Fubini-Study metric \cite{Dhillon2008,Gluck1967,Love2005}
is an angle whose cosine (squared, in the complex case) measures volume contraction in orthogonal projections between subspaces.
It extends to an asymmetric metric on the full Grassmannian, given by an asymmetric angle \cite{Mandolesi_Grassmann}
with better properties than similar symmetric angles \cite{Gluck1967,Gunawan2005,Jiang1996}.
Links with products of Grassmann and Clifford algebras \cite{Dorst2007,Mandolesi_Products,Mandolesi_Contractions} give formulas to compute it.
Angles with orthogonal complements can be used to gain extra information.

Our focus on the asymmetric Fubini-Study metric was motivated by use in quantum information theory, but its useful properties make it well suited for general applications involving subspaces of various dimensions.
We also show how other Grassmannian metrics can be extended to asymmetric metrics on the full Grassmannian.
In particular, the Binet-Cauchy metric \cite{Hamm2008} becomes the sine of the asymmetric angle.

\Cref{sc:preliminaries} sets up notation and reviews concepts and results. 
We study the asymmetric angle in \Cref{sc:asymmetric angle}, and related angles in \Cref{sc:Related angles}.
\Cref{sc:Other asymmetric metrics} gives a method to obtain asymmetric metrics.
\Cref{sc:conclusion} closes with a few remarks.
Some inequalities are proven in \Cref{sc:Distance inequalities}.

\section{Preliminaries}\label{sc:preliminaries}

In this article $X=\F^n$ for $\F = \R$ or $\C$, with inner product  $\inner{\cdot,\cdot}$ (Hermitian product if $\F=\C$, conjugate linear in the left  entry).
A \emph{$p$-subspace} is a $p$-dimensional subspace, 
a \emph{line} is a 1-subspace, 
$\F v=\{cv:c\in\F\}$ for $v\in X$, and
$U(X)$ is the unitary group of $X$ (orthogonal group, if $\F=\R$).
For a subspace $V$,
$P_V: X\rightarrow V$ is the orthogonal projection, and
its underlying real space is $V_\R$ ($=V$ if $\F=\R$),
with inner product $\operatorname{Re}\inner{\cdot,\cdot}$.
We use the term `distance' loosely, and `metric' in the formal sense (of metric spaces).
When we say `projected', it means `orthogonally projected'.

\emph{Euclidean} and \emph{Hermitian angles} for nonzero $v,w\in X$ are, respectively, 
$\theta_{v,w} = \cos^{-1}\frac{\operatorname{Re}\inner{v,w}}{\|v\| \|w\|} \in[0,\pi]$ and 
$\gamma_{v,w} = \cos^{-1}\frac{|\inner{v,w}|}{\|v\| \|w\|} \in[0,\frac{\pi}{2}]$. 
Also, let $\theta_{0,v} = \theta_{0,0} = \gamma_{0,v} =\gamma_{0,0} = 0$ and $\theta_{v,0} = \gamma_{v,0}  = \frac{\pi}{2}$.
If $\F=\R$, $\theta_{v,w}$ is the usual angle and $\gamma_{v,w} = \min\{\theta_{v,w},\pi-\theta_{v,w}\}$.
If $\F=\C$, $\theta_{v,w}$ is the usual angle in $X_\R$ 
and $\gamma_{v,w} = \theta_{v,P_{\C w} v}$ \cite{Scharnhorst2001}.
If $\inner{v,w} \geq 0$ (so $\theta_{v,w} = \gamma_{v,w}$) we say $v$ and $w$ are \emph{aligned}, what happens when $P_{\F w} v= \lambda w$ for $\lambda \geq 0$.
The same definitions will later apply to multivectors.

For $q=0,1,2,\ldots$ let $\II^q = \bigcup_{p=0}^q \II_p^q$ with $\II_0^q=\{\emptyset\}$ and, for $1\leq p\leq q$, $\II_p^q= \{ (i_1,\ldots,i_p)\in\N^p : 1\leq i_1 < \cdots<i_p\leq q \}$.
We also write a multi-index $(i_1,\ldots,i_p)$ as $i_1\cdots i_p$.
For $\ii,\jj\in\II^q$, we form $\ii \cup \jj, \ii \cap \jj, \ii-\jj \in \II^q$ from the union and intersection of their indices, 
and by removing from $\ii$ any indices of $\jj$.
We write $\ii\subset \jj$ if all indices of $\ii$ are in $\jj$. 
If $\ii \cap \jj = \emptyset$, $\epsilon_{\ii\jj}$ is the sign of the permutation that puts $\ii\jj$ (the indices of $\ii$ followed by those of $\jj$) in increasing order.	
For $\ii\in\II_p^q$, let $\ii' = (1,\ldots, q)-\ii \in\II_{q-p}^q$.

\subsection{Grassmann algebra}

Grassmann's exterior algebra \cite{Browne2012,Winitzki2010,Yokonuma1992} is a natural formalism for working with subspaces. 
It is a graded algebra $\bigwedge X = \bigoplus_{p\in\Z} \bigwedge^p X$ of \emph{multivectors}, with $\bigwedge^0 X = \F$, $\bigwedge^1 X = X$ and $\bigwedge^p X = \{0\}$ for $p \not\in [0,n]$. It has a bilinear associative \emph{exterior product} $\wedge : \bigwedge^p X \times \bigwedge^q X \rightarrow \bigwedge^{p+q} X$ such that, for $\kappa,\lambda \in \F$ and $u,v\in X$, $\kappa\wedge\lambda = \kappa\lambda$, $\lambda \wedge v = \lambda v$ and $u \wedge v = -v\wedge u$  (so $v\wedge v=0$).
If $A\in \bigwedge^p X$ and $B\in \bigwedge^q X$, $A\wedge B = (-1)^{pq} B\wedge A$.
Elements of $\bigwedge^p X$ have \emph{grade} $p$ and are linear combinations of \emph{$p$-blades} $B=v_1\wedge\cdots\wedge v_p$ for $v_1,\ldots,v_p\in X$. 
If $B\neq 0$, it \emph{represents} a $p$-subspace $[B]=\Span\{v_1,\ldots,v_p\} = \{v\in X: v\wedge B=0\}$.
A scalar $\lambda\in\bigwedge^0 X$ is a $0$-blade, and $[\lambda] = \{0\}$ (but we only say it represents $\{0\}$ if $\lambda\neq 0$).
	\SELF{Mas só tem $V=\Ann(\lambda)$ e $\bigwedge^0 V=\Span\{\lambda\}$ se  $\lambda \neq 0$}

The inner product of $A=v_1\wedge\cdots\wedge v_p$ and $B=w_1\wedge\cdots\wedge w_p$ is $\inner{A,B} = \det\!\big(\inner{v_i,w_j}\big)$, and $\inner{\kappa,\lambda}=\bar{\kappa}\lambda$ for  $\kappa,\lambda \in\bigwedge^0 X$.
It is extended linearly (sesquilinearly, if $\F=\C$) with distinct $\bigwedge^p X$'s being  orthogonal.
If $\F=\R$, $\|A\|=\sqrt{\inner{A,A}}$ is the $p$-volume of the parallelotope spanned by $v_1,\ldots,v_p$ (Fig.\,\ref{fig:blades}).
If $\F=\C$, $\|A\|^2$ is the $2p$-volume of the parallelotope spanned by $v_1,\im v_1,\ldots,v_p, \im v_p$.
If $[A] \perp [B]$, $\|A\wedge B\| = \|A\|\|B\|$.
For $u,v \in X$, $\|u\wedge v\| = \|u\|\|v\|\sin \gamma_{u,v}$.

\begin{figure}
	\centering
	\includegraphics[width=0.75\linewidth]{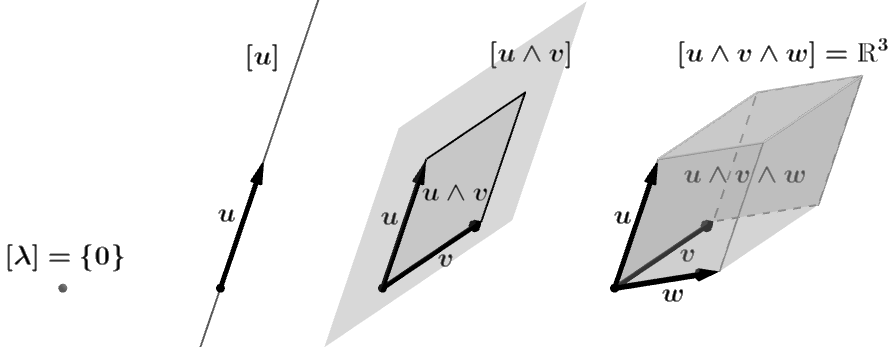}
	\caption{In $\bigwedge \R^3$, a 0-blade is a scalar $\lambda$, a 1-blade is a vector $u$, a 2-blade is shown as a parallelogram $u\wedge v$, and a 3-blade as a parallelepiped $u\wedge v\wedge w$. Their spaces are a point $[\lambda]=\{0\}$, a line $[u]$, a plane $[u\wedge v]$ and the space $[u\wedge v\wedge w] = \R^3$. Their norms are $|\lambda|$, the length of $u$, the area of $u\wedge v$ and the volume of $u\wedge v\wedge w$.}
	\label{fig:blades}
\end{figure}

We write just $P_A$ for $P_{[A]}$, and, given a subspace $V\subset X$, use $P_V$ for the orthogonal projection $P_{\bigwedge V}:\bigwedge X\rightarrow \bigwedge V$, as it extends $P_V:X\rightarrow V$ via $P_V 1 = 1$ and $P_V(A\wedge B) = P_V A \wedge P_V B$ for $A,B\in \bigwedge X$.

Given $v_1,\ldots,v_q \in X$, we write $v_\ii = v_{i_1}\wedge\cdots\wedge v_{i_p}$ for $\ii=(i_1,\ldots,i_p)\in\II_p^q$, and $v_\emptyset = 1$.
If $\beta = (v_1,\ldots,v_q)$ is a basis of $V$, $\{v_\ii\}_{\ii\in\II_p^q}$ and $\{v_\ii\}_{\ii\in\II^q}$ are bases of $\bigwedge^p V$ and $\bigwedge V$, and are orthonormal if $\beta$ is.
As $\dim \bigwedge^p V = \binom{q}{p}$, $\bigwedge^q V = \Span\{v_1\wedge\cdots\wedge v_q\}$ is a line in $\bigwedge X$.
Each $[v_\ii]$ is a \emph{coordinate subspace}.
For $\ii,\jj\in\II^q$, 
if $\ii\cap\jj = \emptyset$ then
$v_{\ii} \wedge v_{\jj} = \epsilon_{\ii\jj} v_{\ii \cup \jj}$, otherwise $v_{\ii} \wedge v_{\jj} = 0$.
For nonzero blades, if $[A]\cap[B] = \{0\}$ then $[A\wedge B] = [A]\oplus [B]$, otherwise $A\wedge B = 0$.
For any multivectors, with those bases one obtains:

\begin{lemma}\label{pr:wedge disjoint multivectors}
	Let $A\in\bigwedge V$ and $B\in\bigwedge W$ for disjoint
	subspaces $V,W\subset X$. Then $A\wedge B = 0 \Leftrightarrow A=0$ or $B=0$.
	\SELF{$(v_i),(w_j)$ bases $V,W$ \\ $\Rightarrow$ $(v_i,w_j)$ basis $V\oplus W$. \\[3pt]
		$(v_\ii)_{\ii\in\cup \I_i^p}$ basis $\bigwedge V$, \\ $(w_\jj)_{\jj\in\cup \I_j^q}$ basis $\bigwedge W$, \\ $(v_\ii\wedge w_\jj)$ basis $\bigwedge(V\oplus W)$. \\[3pt]
		$A=\sum a_\ii v_\ii$, $B=\sum b_\jj w_\jj$, \\$A\wedge B=\sum a_\ii b_\jj v_\ii\wedge w_\jj$. \\[3pt]
		$A\wedge B=0$  $\Rightarrow$ $a_\ii b_\jj=0 \ \forall \ii,\jj$ \\ 
		Some $a_\ii\neq 0$ $\Rightarrow$ $b_\jj=0\ \forall \jj$, e vice versa.} 
\end{lemma}

\subsection{Algebraic varieties}\label{sc:Projective spaces and Grassmannians}


The \emph{projective space} \cite{Griffiths1994} of $X$ is $\PR(X) = \{$lines of $X\}$.
\emph{Angular, chordal} and \emph{gap distances} for $K=\Span\{u\}$ and $L=\Span\{v\}$ are  
$\theta_{K,L} = \gamma_{u,v}$, 
$c_{K,L} = \|u-v\| = 2\sin\frac{\gamma_{u,v}}{2}$ and 
$g_{K,L} = \|u-P_{L}u\| = \sin \gamma_{u,v}$, respectively (Fig.\,\ref{fig:distances}).
As argued in \cite{Qiu2005}, $\theta_{K,L}$ is more fundamental, since $c_{K,L}$ and $g_{K,L}$ derive from it via concave functions, 
and while for these the triangle inequality attains equality only in trivial cases, 
	\SELF{if 2 of the lines coincide}
$\PR(X)$ is a geodesic metric space with the \emph{Fubini-Study metric} 
	\CITE{Goldman1999} 
$d_{FS}(K,L) = \theta_{K,L}$.

\begin{figure}
	\centering
	\includegraphics[width=0.35\linewidth]{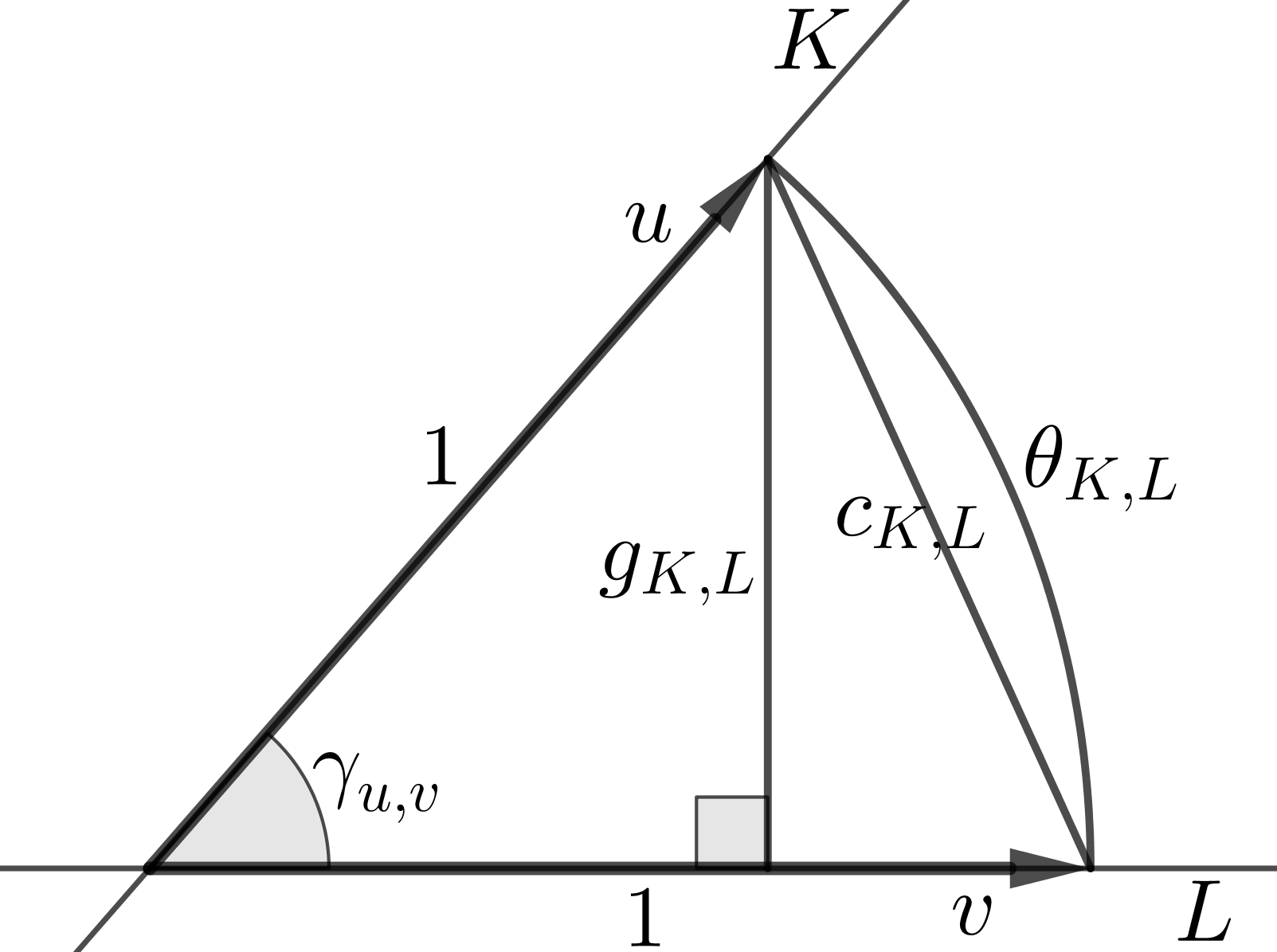}
	\caption{Angular, chordal and gap distances between lines ($u$ and $v$ aligned)}
	\label{fig:distances}
\end{figure}

We prove the following spherical triangle inequality in \Cref{sc:Distance inequalities},
\CITE{real: Reid2005 p.38 \\
	complex: Goldman1999 p.16, em termos de Fubini}
as usual proofs omit details we need when $\F=\C$.
Of course, the result also holds in $\PR(\bigwedge X)$ with $u,v,w$ being multivectors.

\begin{proposition}\label{pr:spherical triangle inequality lines}
	$\theta_{J,L} \leq \theta_{J,K} + \theta_{K,L}$ for $J,K,L \in \PR(X)$, with equality if, and only if, $J=\Span\{u\}$, $K=\Span\{v\}$, $L=\Span\{w\}$ 
	\SELF{deixar $\Span\{u\}$ ao invés de $[u]$ para não confundir quando  usar com lines spanned by blades}
	for aligned $u,v,w\in X$ 
	with $v=\kappa u + \lambda w$ for $\kappa,\lambda \geq 0$.
\end{proposition}

Equality means $K$ is in a minimal geodesic between $J$ and $L$.
It happens when $u,v,w$ are in an isotropic
	\CITE{\emph{totally real} {Goldman1999}, \emph{antiholomorphic} {Rosenfeld1997}, but some authors use these differently}
real plane (a real plane where $\inner{\cdot,\cdot}$ is real),
with $\R v$ in the smaller pair of angles formed by $\R u$ and $\R w$.
The next example shows the need for alignment.

\begin{example}
	In $X=\C^2$, let $u=(\im,0)$, $w=(\frac{1}{2},\frac{\sqrt{3}}{2})$ and $v=u+w$.
	As $\gamma_{u,w}=60^\circ$ and $\gamma_{u,v}=\gamma_{v,w}\cong 38^\circ$, 
	$\C v$ is not in a geodesic segment between $\C u$ and $\C w$ in $\PR(X)$. 
	But $\theta_{u,w}=90^\circ$ and $\theta_{u,v}=\theta_{v,w}=45^\circ$, 
	so $\R v$ lies in a geodesic segment between $\R u$ and $\R w$ in $\PR(X_\R)$.
\end{example}

For a $q$-subspace $V$, the \emph{Grassmannian} $G_p(V) = \{p$-subspaces of $V\}$ is a compact manifold \cite{Griffiths1994,Kozlov2000III,Kozlov2000I}.
	\OMIT{the compact group $U(V)$ acts transitively}
For $p>q$, $G_p(V)=\emptyset$.
The \emph{full Grassmannian} is
$G(V)=\bigcup_{p=0}^q G_p(V) = \{$subspaces of $V\}$. 
The \emph{\Plucker\  embedding} $G(X)\hookrightarrow \PR(\bigwedge X)$ maps $U \in G_p(X)$ to its line $\bigwedge^p U \in \PR(\bigwedge X)$,
and lets us identify $G(X)$ with the set of nonzero blades modulo scalar multiplication.

\subsection{Principal angles and partial orthogonality}

Principal or canonical angles \cite{Afriat1957,Bjorck1973,Qiu2005} 
	\CITE{Gluck1967,Golub2013}
are widely used in the study of Grassmannians and other areas.

\begin{definition}\label{df:principal}
	For nonzero $V,W \in G(X)$, orthonormal bases $\beta_V = (e_1,\ldots,e_p)$ and $\beta_W= (f_1,\ldots,f_q)$ are \emph{associated principal bases}, formed by \emph{principal vectors}, if $\inner{e_i,f_j} = 0$ for $i\neq j$ and $\inner{e_i,f_i} = \cos\theta_i$ for $1\leq i \leq m=\min\{p,q\}$ and \emph{principal angles} $0\leq \theta_1\leq\cdots\leq\theta_m\leq\frac \pi 2$.
\end{definition}

We also say $\beta_W$ is a principal basis of $W$ \wrt $V$.
Note that $\theta_i = \theta_{e_i,f_i} = \gamma_{e_i,f_i}$.
The number of null $\theta_i$'s is $\dim (V\cap W)$. 

If $\sigma_1\geq\cdots\geq\sigma_m$ are the singular values of the orthogonal projection $P:V\rightarrow W$ then $\cos\theta_i = \sigma_i$.
So the $\cos^2\theta_i$'s are the eigenvalues of $P^*P$ if $p\leq q$, or $PP^*$ if $p>q$, 
while the $e_i$'s and $f_i$'s are orthonormal eigenvectors of $P^*P$ and $PP^*$, respectively.
The $\theta_i$'s are uniquely defined, but the $e_i$'s and $f_i$'s are not.

A recursive description is that $e_1$ and $f_1$ form the minimal angle $\theta_1 = \min\left\{ \theta_{v,w} : 0\neq v\in V, 0\neq w\in W \right\}$; in their orthogonal complements we find $e_2$, $f_2$ and $\theta_2$ in the same way; and so on.
For $i>m$ other vectors are chosen to complete an orthonormal basis.
Geometrically, the unit sphere of $V$ projects to an ellipsoid in $W$,
and if $\F=\R$ the $e_i$'s for $1\leq i\leq m$ project onto its semi-axes, of lengths $\cos\theta_i$, and the $f_i$'s point along them. If $\F=\C$, for each $i$ there are two semi-axes of equal lengths where $e_i$ and $\im e_i$ project, so in the underlying real spaces each $\theta_i$ is twice repeated.

\begin{example}\label{ex:real principal angles}
	Let $(f_1,\ldots,f_5)$ be the canonical basis of $\R^5$. Then $e_1=\frac{f_1+f_3}{\sqrt{2}}$, $e_2=\frac{f_2+f_4}{\sqrt{2}}$, $f_1$, $f_2$ and $f_5$ are principal vectors for $V=[e_{12}]$ and $W=[f_{125}]$, with principal angles $\theta_1=\theta_2=45^\circ$.
\end{example}

\begin{example}\label{ex:complex principal angles}
	In $\C^4$, we have $e_1=(\frac{\sqrt{2}}{2},\frac{\sqrt{2}}{2},0,0)$, $e_2=(0,0,\frac{\im}{2},\frac{\sqrt{3}}{2})$, $f_1=(\frac{1+\im}{2},\frac{1-\im}{2},0,0)$ and $f_2=(0,0,\im ,0)$ as principal vectors for $V=[e_{12}]$ and $W=[f_{12}]$, with principal angles $\theta_1=45^\circ$ and $\theta_2=60^\circ$.
	The underlying real subspaces $V_\R,W_\R\subset\R^8$ have principal vectors
	\[\begin{aligned}
		e_1 &= (\tfrac{\sqrt{2}}{2},0,\tfrac{\sqrt{2}}{2},0,0,0,0,0), &f_1 &= (\tfrac12,\tfrac12,\tfrac12,-\tfrac12,0,0,0,0), \\
		\tilde{e}_1 & = \im e_1 = (0,\tfrac{\sqrt{2}}{2},0,\tfrac{\sqrt{2}}{2},0,0,0,0), &\tilde{f}_1 &= \im f_1 = (-\tfrac12,\tfrac12,\tfrac12,\tfrac12,0,0,0,0), \\
		e_2 &= (0,0,0,0,0,\tfrac{1}{2},\tfrac{\sqrt{3}}{2},0), &f_2 &= (0,0,0,0,0,1,0,0),\\
		\tilde{e}_2 &= \im e_2 = (0,0,0,0,-\tfrac{1}{2},0,0,\tfrac{\sqrt{3}}{2}), \quad&\tilde{f}_2 &= \im f_2 = (0,0,0,0,-1,0,0,0), \end{aligned}\]
	with principal angles $\theta_1=\tilde{\theta}_1=45^\circ$ and $\theta_2=\tilde{\theta}_2=60^\circ$. 
\end{example}

We will need the following results.
The first one is an easy generalization of \cite[Thm. 3]{Wong1967}, and shows a unitarily invariant distance between subspaces must be a function of the dimensions and principal angles \cite{Qiu2005,Stewart1990},
all of which are needed to fully describe their relative position.
 
 \begin{proposition}[\cite{Wong1967}]\label{pr:relative position}
 	For nonzero $V,V'\in G_p(X)$ and $W,W'\in G_q(X)$, there is $T\in U(X)$ with $T(V)=V'$ and $T(W)=W'$ if, and only if, $V$ and $W$ have the same principal angles as $V'$ and $W'$.
 \end{proposition}

 \begin{proposition}[\cite{Qiu2005}]\label{pr:thetas perps}
 	For $V,W\neq \{0\}$ or $X$, the nonzero principal angles of $V^\perp$ and $W^\perp$ are the same as those of $V$ and $W$.
 \end{proposition}

\begin{lemma}\label{pr:P ei}
	With the notation of \Cref{df:principal}, let $A = e_1 \wedge\cdots\wedge e_p$.
	\begin{enumerate}[i)]
		\item If $V\not\perp W$ then $P_W(V) = [f_1\wedge\cdots\wedge f_k]$ for $k = \max\{i:\theta_i\neq \frac{\pi}{2}\}$, and the principal angles of $V$ and $P_W(V)$ are $\theta_1,\ldots,\theta_k$. \label{it:PV}
		\item $P_W A = \cos\theta_1\cdots\cos\theta_p\, f_1\wedge\cdots\wedge f_p$ if $p\leq q$, otherwise $P_W A = 0$.\label{it:P e1...ep}
		\item $[P_W A] = P_W(V) \Leftrightarrow P_W A \neq 0$. \label{it:[P_W A] = P_W[A]}
		\SELF{requer $W\neq\{0\}$}
	\end{enumerate}
\end{lemma}
\begin{proof}
	$P_W e_i = f_i\,\cos\theta_i$ for $1\leq i\leq m$, and $P_W e_i = 0$ if $i>m$.
	\OMIT{$\inner{A.B} = \det(\inner{e_i,f_j}) = \prod_i \inner{e_i,f_i} = \prod_i \cos\theta_i$}
\end{proof}

\begin{definition}
	For $V,W \in G(X)$, when $W^\perp \cap V \neq \{0\}$ we say $V$ is \emph{partially orthogonal} to $W$, and write $V\pperp W$.
\end{definition}

If $p=q$ then $V \pperp W  \Leftrightarrow W\pperp V$. In general the relation is asymmetric.

\begin{proposition}\label{pr:pperp}
	For $V,W\in G(X)$, $V\pperp W$ is equivalent to:
	\begin{enumerate}[i)]
	\item $\dim P_W(V)< \dim V$. 
	\item $\dim W< \dim V$ or a principal angle is $\frac{\pi}{2}$.\label{it:pperp q<p ou pi2}
	\item $P_W A = 0$ for a blade $A$ representing $V$.
	\end{enumerate} 
\end{proposition} 
\begin{proof}
	For $V,W\neq\{0\}$ it follows from \Cref{pr:P ei}.
		\OMIT{$(i \Leftrightarrow ii)$ Immediate \\
			$(ii \Leftrightarrow iii)$ \ref{pr:P ei}\ref{it:PV} \\
			$(iii \Leftrightarrow iv)$ \ref{pr:P ei}\ref{it:P e1...ep} }
\end{proof}

\begin{definition}\label{df:PO}
	For $V\in G_p(X)$ and $W\in G_q(X)$, 
	decompose $W = W_P \oplus W_\perp$ with $W_P=\{0\}$ and $W_\perp = W$ if $m = \min\{p,q\} = 0$, otherwise
	$W_P=[f_1 \wedge\cdots\wedge f_m]$ and $W_\perp=[f_{m+1}\wedge\cdots\wedge f_q]$ ($=\{0\}$ if $m=q$) for a principal basis $(f_1,\ldots,f_q)$ of $W$ \wrt $V$.
	We refer to $W_P$ as a \emph{projective subspace} of $W$ \wrt $V$.
\end{definition}

If $V\not\pperp W$ then $W_P=P_W(V)$ and $W_\perp = V^\perp \cap W$,
otherwise they depend on the principal basis, with $W_P \supset P_W(V)$ and $W_\perp \subset V^\perp \cap W$ (strict inclusions if $p \leq q$).
	\OMIT{\ref{pr:P ei}\ref{it:PV}, \Cref{pr:pperp}\ref{it:pperp q<p ou pi2}}
The following lemma is immediate:

\begin{lemma}\label{pr:thetas WP Wperp}
	Nonzero $V,W\in G(X)$ have the same principal angles as $V$ and $W_P$, and the same nonzero principal angles as $V\oplus W_\perp$ and $W$. 
\end{lemma}

\subsection{Asymmetric metrics}\label{sc:Asymmetric metrics}

Asymmetric metrics \cite{Anguelov2016,Mennucci2013,Mennucci2014} do not require $d(x,y)=d(y,x)$,
and appear for example in directed graphs or Finsler manifolds, generalizing metrics as Finsler geometry generalizes the Riemannian one \cite{Flores2013}.
Other terms are quasi-metric or quasi-distance \cite{Albert1941,Cobzas2012,Flores2013} and $T_0$-quasi-pseudometric \cite{Kazeem2014,Kunzi2009}.
Ref.\ \cite{Kunzi2001} reviews the subject from a topological perspective.

\begin{definition}\label{df:asymmetric metric}
	An \emph{asymmetric metric}  on a non-empty set $M$ is a function $d:M\times M\rightarrow [0,\infty]$ satisfying, for all $x,y,z\in M$:
	\begin{enumerate}[i)]
		\item \emph{Separation condition}: $d(x,y)=d(y,x)=0 \Leftrightarrow x=y$. \label{it:d=0 x=y}
		\item \emph{Oriented triangle inequality}: $d(x,z) \leq d(x,y)+d(y,z)$. \label{it:oriented triang ineq}
	\end{enumerate}
\end{definition}

The order of $x,y,z$ is important in \ref{it:oriented triang ineq}.
Instead of \ref{it:d=0 x=y}, some authors \cite{Chenchiah2009,Wilson1931} use $d(x,y)=0 \Leftrightarrow x=y$, which is too restrictive for some purposes: e.g., intuitively, the distance from a line to a plane containing it is $0$.

With the \emph{conjugate} distance $d^-(x,y) = d(y,x)$, we have \emph{max} and \emph{mean symmetrized} metrics $\hat{d}=\max\{d,d^-\}$ and $\bar{d} = \frac{d+d^-}{2}$.
	\CITE{cite[p.7]{Deza2016}}
Symmetrizing by the $\min$ does not preserve the triangle inequality.
Some commonly used metrics, like the gap \cite{Kato1995} and the Hausdorff distance \cite{Huttenlocher1993}, 
	\CITE{Knauer2011 \\ Blumenthal1970 p.\,24 proves oriented triangle ineq for directed Hausdorff (inside proof for Hausdorff)}
are symmetrized versions of asymmetric metrics.

Asymmetric metrics give a choice of \emph{backward, forward} or \emph{symmetric topologies} 
	\CITE{Flores2013, Mennucci2013, Mennucci2014}
$\tau^-$, $\tau^+$, $\tau$ generated, respectively, by \emph{backward balls} $B^-_r(x) = \{y\in M:d(y,x)<r\}$, \emph{forward balls} $B^+_r(x) = \{y\in M:d(x,y)<r\}$, or \emph{symmetric balls} $B_r(x) = B^+_r(x) \cap B^-_r(x)$.
	\SELF{$= \{y\in M:\hat{d}(x,y)<r\}$}
While $\tau$ is Hausdorff, as it is the metric topology of $\hat{d}$, $\tau^\pm$ are just $T_0$.
	\OMIT{have $d(x,y)=0$ for $x\neq y$}
Also, $d$ is continuous in $\tau$ 
	\OMIT{$d(x',y)-d(x,y)\leq d(x',x)$, $d(x,y)-d(x',y)\leq d(x,x')$, $|d(x',y)-d(x,y)|\leq \hat{d}(x,x')$ and likewise with $y'$}
but not in $\tau^\pm$,
	\SELF{in both $\tau^\pm$, $d$ can be discontin in both variables}
as \ref{it:oriented triang ineq} gives
$\max\{d(x,y) - d(x,z) , d(z,x) - d(y,x)\} \leq d(z,y)$ 
instead of $|d(x,y)-d(x,z)| \leq d(y,z)$.
On the other hand, there are more continuous paths with $\tau^\pm$ than $\tau$: in $\tau^-$ (resp. $\tau^+$), $p:\R\rightarrow M$ is continuous at $t\in\R$ if, given $\epsilon>0$, there is $\delta>0$ with $d(p(s),p(t))<\epsilon$ (resp. $d(p(t),p(s))<\epsilon$) for all $s\in (t-\delta,t+\delta)$, while $\tau$ requires both conditions.

\subsection{Metrics on Grassmannians}\label{sc:Distances on Grassmannians}

Usual metrics on $G_p(X)$ \cite{Edelman1999,Qiu2005,Stewart1990,Ye2016} are obtained by embedding it in other metric spaces, or as geodesic distances,
but are ultimately based on the distances of Fig.\,\ref{fig:distances}, fitting into the scheme of \Cref{tab:metrics same dim}.

Let $V,W\in G_p(X)$ have principal angles $\theta_1\leq\cdots\leq\theta_p$ and associated principal bases $(e_1,\ldots,e_p)$ and $(f_1,\ldots,f_p)$.
Also, let $\mathbf{E}$ and $\mathbf{F}$ be $n\times p$ matrices formed with these vectors, $A=e_1\wedge\cdots\wedge e_p$, $B=f_1\wedge\cdots\wedge f_p$,
$\|\cdot\|_F$ be the Frobenius norm, and $\|\cdot\|_{\text{op}}$ the operator norm.
We have:

\begin{table}[]
	\centering
	\renewcommand{\arraystretch}{1}
	\begin{tabular}{cccc}
		\toprule
		Type & angular & chordal & gap \\
		\cmidrule(lr){1-1} \cmidrule(lr){2-2} \cmidrule(lr){3-3} \cmidrule(lr){4-4}
		$l^2$ & \makecell{geodesic \\ $d_g = \sqrt{\sum_{i=1}^p \theta_{[e_i],[f_i]}^2}$} & \makecell{chordal Frobenius \\ $d_{cF} = \sqrt{\sum_{i=1}^p c_{[e_i],[f_i]}^2}$} & \makecell{projection Frobenius \\ $d_{pF} = \sqrt{\sum_{i=1}^p g_{[e_i],[f_i]}^2}$}
		\\		
		\cmidrule(lr){1-1} \cmidrule(lr){2-2} \cmidrule(lr){3-3} \cmidrule(lr){4-4}
		$\wedge$ & \makecell{ Fubini-Study \\ $d_{FS} = \theta_{\bigwedge^p V,\bigwedge^p W}$} & \makecell{chordal-$\wedge$ \\ $d_{c\wedge} = c_{\bigwedge^p V,\bigwedge^p W}$} & \makecell{Binet-Cauchy \\ $d_{BC} = g_{\bigwedge^p V,\bigwedge^p W}$}
		\\
		\cmidrule(lr){1-1} \cmidrule(lr){2-2} \cmidrule(lr){3-3} \cmidrule(lr){4-4}
		max & \makecell{Asimov \\ $d_{A} = \theta_{[e_p],[f_p]}$} & \makecell{chordal 2-norm \\ $d_{c2} = c_{[e_p],[f_p]}$} & \makecell{projection 2-norm \\ $d_{p2} = g_{[e_p],[f_p]}$}
		\\
		\bottomrule
	\end{tabular}
	\caption{Metrics on $G_p(X)$ in terms of the angular, chordal or gap distances between principal lines $[e_i]$ and $[f_i]$ of $V$ and $W$, or between lines $\bigwedge^p V$ and $\bigwedge^p W$ in $\bigwedge X$}
	\label{tab:metrics same dim}
\end{table} 

\begin{enumerate}[1)]
	\item \emph{$l^2$ metrics}: given by the $l^2$ norm of the vector formed by the angular, chordal or gap distances of principal lines $[e_i]$ and $[f_i]$.
	\begin{enumerate}[a)]
		\item \emph{Geodesic} \cite{Dhillon2008,Kozlov2000III,Wong1967}:
			\CITE{Zhang2018,Zuccon2009}
		$d_g = \sqrt{\sum_{i=1}^p \theta_i^2}$ is the canonical metric, being the geodesic distance for the unique\footnote{Except in $G_2(\R^4)$ \cite[p.\,2249]{Kozlov2000I}, \cite[p.\,591]{Wong1967}.} 
		(up to scaling) Riemannian metric invariant by $U(X)$, obtained by identifying the tangent space at $V$ with $\Hom(V,V^\perp)$ and using the Hilbert-Schmidt product.
			\SELF{$\inner{S,T} = \Tr(S^*T)$ \\ Bendokat2020 com fator $\frac12$}
		Also called \emph{Grassmann distance} \cite{Deza2016,Ye2016}.
			\CITE{Used in Lerman2011}
			\SELF{max $\frac{\pi\sqrt{p}}{2}$, when $V\perp W$  (if $X$ large enough)}
		
		\item \label{it:cF}
		\emph{Chordal Frobenius} \cite{Edelman1999}: 
			\CITE{$\rho_b$ in {Stewart1990} p.95,99. \\ Paige1984}
		embedding $G_p(X)$ in $\F^{n\times p}$ with $\|\cdot\|_F$,
			\SELF{decomp in orthon basis} 
		$d_{cF}  = \|\mathbf{E}-\mathbf{F}\|_F = \sqrt{\sum_{i=1}^p \|e_i - f_i\|^2} = 2\sqrt{\sum_{i=1}^p \sin^2 \frac{\theta_i}{2}}$.
			\SELF{$= \sqrt{2p-2\sum_{i=1}^p \cos\theta_i}$; max $\sqrt{2p}$, if $V\perp W$}
		Also called \emph{Procrustes distance} \cite{Chikuse2012,Hamm2008,Turaga_2008}. 
		
		\item \label{it:pF}
		\emph{Projection Frobenius} \cite{Draper2014,Hamm2008}: 
		embedding $G_p(X)$ in the set of projection matrices with $\|\cdot\|_F$, we find
		$d_{pF} = \frac{1}{\sqrt{2}} \|P_V-P_W\|_F = \sqrt{\sum_{i=1}^p \|e_i - P_W e_i\|^2} = \sqrt{\sum_{i=1}^p \sin^2 \theta_i}$.
			\SELF{max $\sqrt{p}$, when $V\perp W$}
		Often called \emph{chordal distance}, due to another embedding in a sphere \cite{Barg2002,Conway1996,Dhillon2008,Ye2016}.
			\CITE{\cite{Deza2016}: Frobenius distance is $\|P_V-P_W\|_F$}
	\end{enumerate}

	\item \emph{$\wedge$ metrics}: obtained via the \Plucker\ embedding, with the angular, chordal or gap distance of $\bigwedge^p V = \Span\{A\}$ and $\bigwedge^p W = \Span\{B\}$.
	\begin{enumerate}[a)]
		\item \emph{Fubini-Study} \cite{Dhillon2008,Love2005,Stewart1990}: 
		$d_{FS} =  \gamma_{A,B} = \cos^{-1}(\prod_{i=1}^p \cos\theta_i)$.
			\CITE{$\rho_\theta$ in {Stewart1990} p.96,99. Parece surgir em Lu1963}
			\SELF{max $\frac{\pi}{2}$, when $V\pperp W$}
		It is a geodesic distance through the ambient space $\PR(\bigwedge^p X)$, and an angle that measures volume contraction \cite{Gluck1967}.

		\item \emph{chordal-$\wedge$}: $d_{c\wedge} = \|A-B\| = 2\sin \frac{\gamma_{A,B}}{2} = \sqrt{2-2\prod_{i=1}^p \cos\theta_i}$.
		This metric does not seem to have been considered before.
		
		\item \emph{Binet-Cauchy} \cite{Hamm2008,Vishwanathan2006,Wolf2003}:
		 $d_{BC} = \|A-P_W A\| = \sin\gamma_{A,B} = \sqrt{1-\prod_{i=1}^p \cos^2\theta_i}$.
			\SELF{$= \|\Psi_A - P_{\Psi_B} \Psi_A\|$, \emph{Grassmann vector}  $\Psi_A$ {Wolf2003} formed by \Plucker\ coord in orthonormal basis;
			max  $1$, when $V\pperp W$}
			\CITE{$\rho_s$ in {Stewart1990} p.96,99. Parece surgir em Lu1963}
	\end{enumerate}

	\item \emph{max metrics}: maximum angular, chordal or gap distance of principal lines  (so, for $[e_p]$ and $[f_p]$).
	\begin{enumerate}[a)]
		\item \emph{Asimov} \cite{Asimov1985}: $d_A = \theta_p$ is the geodesic distance for a Finsler metric 
			\SELF{dada por norma no espaço tangente, simétrica ou não}
		given by $\|\cdot\|_{\text{op}}$ in the tangent space $\Hom(V,V^\perp)$  \cite{Weinstein2000}.
			\SELF{largest angular dist. from $S(V)$ to $S(W)$; max $\frac\pi2$ when $V\pperp W$; cos is largest semi-axis of $P_W(S(V))$}
		
		\item \emph{Chordal 2-norm} \cite{Barg2002}: 
		as in (\ref{it:cF}), but with $\|\cdot\|_{\text{op}}$, we have
		$d_{c2} = \|\mathbf{E}-\mathbf{F}\|_{\text{op}} = \|e_p-f_p\| = 2\sin \frac{\theta_p}{2}$.
			\SELF{$= \max\limits_{v\in S(V)} \min\limits_{w \in S(W)} \|v-w\|$ largest distance from $S(V)$ to $S(W)$ (Hausdorff dist.); max $\sqrt{2}$, when $V\pperp W$}
		The `2-norm' refers to the norm in $X$.
		Also called \emph{spectral distance} \cite{Deza2016,Ye2016}.
		
		\item \emph{Projection 2-norm} \cite{Edelman1999}:
			\CITE{Barg2002,Ye2016}	
		as in (\ref{it:pF}), but with $\|\cdot\|_{\text{op}}$, we have
		$d_{p2} = \|P_V-P_W\|_{\text{op}} = \|e_p - P_W e_p\| = \sin \theta_p = \max\limits_{v\in V, \|v\|=1} \|v-P_W v\|$.
			\SELF{largest dist. from $S(V)$ to $W$; max $1$ when $V\pperp W$; in normed spaces gap is not metric {Kato1995}}
		Also called \emph{gap} \cite{Kato1995,Stewart1990} or \emph{min-correlation} \cite{Hamm2008}.
			\CITE{\emph{containment gap} or \emph{projection distance} in {Deza2016}}
	\end{enumerate}
\end{enumerate}

Inequalities in \Cref{sc:Distance inequalities} show these metrics (Fig.\,\ref{fig:distances Gp}) give the same topology on $G_p(X)$,
and decrease as we move right (if $V \neq W$) or down (if $\dim(V\cap W) < p-1$) in \Cref{tab:metrics same dim}.
For small principal angles, $l^2$ and $\wedge$ metrics converge as\-ymp\-tot\-i\-cal\-ly to $d_g$, so their embeddings are isometric (in the Riemannian sense) \cite{Edelman1999}.

\begin{figure}
	\centering
	\begin{subfigure}[b]{0.325\textwidth}
		\includegraphics[width=\textwidth]{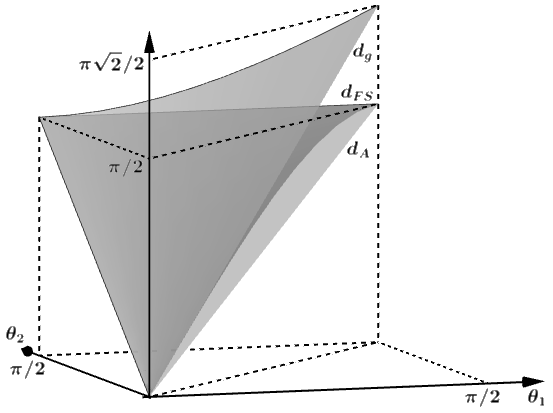}
		\caption{angular metrics}
		\label{fig:spherical triangle inequality_1}
	\end{subfigure}
	\begin{subfigure}[b]{0.325\textwidth}
		\includegraphics[width=\textwidth]{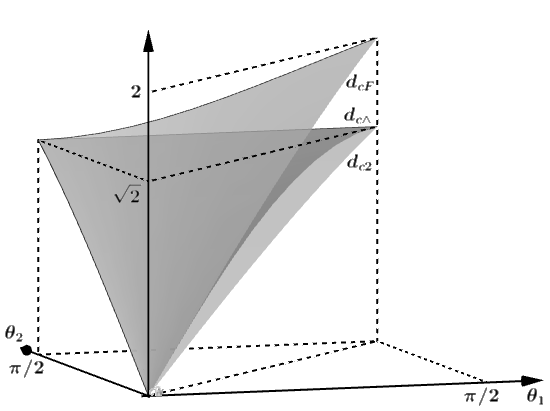}
		\caption{chordal metrics}
		\label{fig:spherical triangle inequality_2}
	\end{subfigure}
	\begin{subfigure}[b]{0.325\textwidth}
		\includegraphics[width=\textwidth]{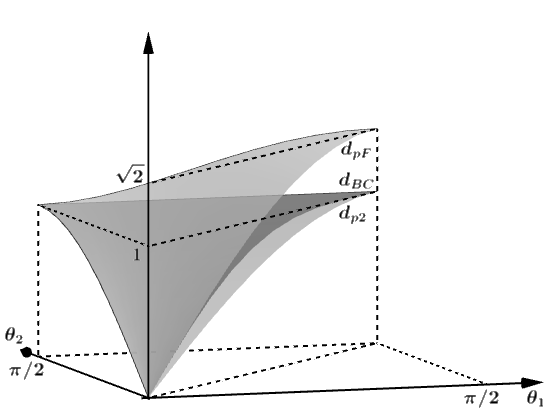}
		\caption{gap metrics}
		\label{fig:spherical triangle inequality_2}
	\end{subfigure}
	\caption{Metrics on $G_2(X)$ as functions of principal angles $\theta_1\leq\theta_2$}
	\label{fig:distances Gp}
\end{figure}

The $l^2$ metrics are maximized when $V \perp W$ (assuming $\dim X \geq 2p$).
	\SELF{To allow $\theta_i=\frac\pi2$ for all $i$}
The $\wedge$ ones, when $V\pperp W$, so once $\theta_p=\frac\pi2$
they ignore the other $\theta_i$'s.
Max metrics also maximize when $V\pperp W$, but always take only $\theta_p$ into account, being unsuitable for applications in which many small differences between subspaces can be more relevant than a single large one.

We note that the following are not metrics on $G_p(X)$:

\begin{itemize}
	\item The \emph{max-correlation} \cite{Hamm2008} or \emph{spectral distance} \cite{Dhillon2008} $d_{mc} = \sin \theta_1$ 
		\SELF{smallest distance from $S(V)$ to $W$; max $1$ when $V\perp W$}
	does not satisfy a triangle inequality, and $d_{mc} = 0 \Leftrightarrow V\cap W \neq \{0\}$.
	
	\item The \emph{Martin metric} for ARMA
		\CITE{Auto-Regressive Moving Average} 
	models \cite{Martin_2000} is presented in \cite{Deza2016,Ye2016} as a metric $d_M = \sqrt{-\log \prod_{i=1}^p\cos^2\theta_i}$ for subspaces.
	But this formula is obtained in \cite{De_Cock_2002} for specific subspaces associated to the models.
		\CITE{spanned by vector of the form $(1,\alpha,\alpha^2,\ldots)$ with $|\alpha|<1$}
	For arbitrary subspaces $d_M$ does not satisfy a triangle inequality (\eg take lines in $\R^2$), and it is $\infty$ when $V\pperp W$.
\end{itemize}

\subsection{Distances on the full Grassmannian}\label{sc:Distances on the full Grassmannian}

Let $V\in G_p(X)$ and $W\in G_q(X)$ have principal angles $\theta_1\leq\cdots\leq\theta_{\min\{p,q\}}$ and associated principal bases $(e_1,\ldots,e_p)$ and $(f_1,\ldots,f_q)$.
On the full Grassmannian $G(X)$ we have the following distances:

\begin{itemize}	
	\item \emph{Full Fubini-Study metric}: obtained via the full \Plucker\ embedding, it extends $d_{FS}$ trivially: as blades of distinct grades are orthogonal, $d_{FS}(V,W) = \frac{\pi}{2}$ whenever $p\neq q$.
	
	\item \emph{Containment gap} \cite{Beattie2005,Kato1995}: asymmetric metric generalizing $d_{p2}$ via 
		\SELF{$\delta(U,W) = \max\limits_{\|u\|=1} \|u-P_W u\| \leq \max\limits_{\|u\|=1} \|u-P_W P_V u\| \leq \max\limits_{\|u\|=1} (\|u-P_V u\| + \|P_V u-P_W P_V u\|) \leq \delta(U,V) + \max\limits_{\|u\|=1} \|\frac{P_V u}{\|P_V u\|}-P_W \frac{P_V u}{\|P_V u\|}\| \leq \delta(U,V) + \max\limits_{\|v\|=1} \|v-P_W v\| = \delta(U,V) + \delta(V,W)$}
	\begin{equation*}
		\delta(V,W) = \max\limits_{v\in V, \|v\|=1} \|v-P_W v\| =
		\begin{cases}
			\sin \theta_p &\text{if } p\leq q, \\
			1 &\text{if } p>q.
		\end{cases}
	\end{equation*}
	Since $\delta(V,W) =0 \Leftrightarrow V\subset W$, it shows how far $V$ is from being contained in $W$.
	Its maximum occurs when $V\pperp W$.
		
	\item \emph{Gap} \cite{Kato1995}: 
		\CITE{Stewart1990 não pois só trata equal dim}
	$\hat{\delta}(V,W) = \max\{\delta(V,W),\delta(W,V)\} = \|P_V - P_W\|_{\text{op}}$.
		\SELF{$\|T\|=\sup\{\frac{\|Tx\|}{\|x\|}\}$ \\ To prove, assume $p=q$, then $\|(P_V - P_W)e_i\| = \|e_i - P_W e_i\|\leq \delta(V,W)$ and $\|(P_V - P_W)e_i^\perp\| = \|P_W e_i^\perp\| = \sin\theta_i \leq \delta(V,W)$}	
	It is a metric extending $d_{p2}$ trivially, with $\hat{\delta}(V,W) = 1$ for $p\neq q$.

	\item \emph{Projection Frobenius} \cite{Basri2011,Draper2014,Pereira2022}: $d_{pF} = \sqrt{\sum_{i=1}^{\min\{p,q\}} \sin^2 \theta_i}$ does not satisfy a triangle inequality (\eg take two lines and their plane).

	\item \emph{Directional distance} \cite{Wang_2006}: generalizes $d_{pF}$ via
		\SELF{Instead of the $e_i$'s one can use any orthonormal basis of $V$.}
		\SELF{$\min\limits_{W'\in\Omega_p^-(W)}\{d_{pF}(V,W')^2\}$ if $p\leq q$, \\	 $\max\limits_{W'\in\Omega_p^+(W)}\{d_{pF}(V,W')^2\}$ if $p>q$}
	\begin{equation*}
		\vec{d}(V,W)^2 = \sum_{i=1}^p \|e_i-P_W e_i\|^2 =
		\begin{cases}
			\sum_{i=1}^p \sin^2 \theta_i &\text{if } p\leq q, \\
			p-q+\sum_{i=1}^q \sin^2 \theta_i &\text{if } p>q.
		\end{cases}
	\end{equation*}
	It is not clear whether it satisfies a triangle inequality. 
	For fixed $p$ and $q$, its minimum is $\sqrt{\max\{0,p-q\}}$, if $V\subset W$ or $W \subset V$.

	\item \emph{Symmetric distance} \cite{Sun_2007,Wang_2006}:
		\CITE{Bagherinia2011, Figueiredo2010, Sharafuddin2010, Zuccon2009}
	$d_s (V,W) = \max\{\vec{d}(V,W),\vec{d}(W,V)\} = \frac{1}{\sqrt{2}} \|P_V-P_W\|_F = \sqrt{|p-q|+\sum_{i=1}^{\min\{p,q\}} \sin^2 \theta_i}$.
	It is a metric.
	For fixed $p$ and $q$, its minimum is $\sqrt{|p-q|}$, if $V\subset W$ or $W\subset V$.
\end{itemize}

For $p\neq q$, $d_{FS}$ and $\hat{\delta}$ have fixed values, so do not give any new information.
As $\delta$ and $\hat{\delta}$ do not take all $\theta_i$'s into account, they are rather rough distances.
Lack of a triangle inequality limits the usefulness of $d_{pF}$ and (possibly) $\vec{d}$.
For $p\neq q$, $d_s$ is a nontrivial metric,
and its nonzero minimum may be useful if subspaces of distinct dimensions must be kept apart.
But if they are similar when one is almost contained in the other,
it is inconvenient to have this expressed by $d_s(V,W) < \sqrt{|p-q|} + \epsilon$, specially if $p$ or $q$ are not known beforehand (\eg if the subspaces are approximate representations obtained by truncating the spectrum of an operator).
	\CITE{Gruber2009}
Other metrics obtained in \cite{Ye2016} have similar problems.

\section{Asymmetric angle}\label{sc:asymmetric angle}

We first describe an angle between subspaces of arbitrary dimensions, and later we show it extends $d_{FS}$ as an asymmetric metric on $G(X)$.

\begin{definition}
	Let $V,W\in G(X)$, and $A$ be a blade representing $V$.
	The \emph{asymmetric angle}\footnote{Formerly called \emph{Grassmann angle} \cite{Mandolesi_Grassmann,Mandolesi_Pythagorean}.} 
	from $V$ to $W$ is $\Theta_{V,W} = \cos^{-1} \frac{\|P_W A\|}{\|A\|} \in[0,\frac{\pi}{2}]$.
		\SELF{$\Theta_{\{0\},W}=0$ for any $W$, \\ $\Theta_{V,\{0\}}=\frac{\pi}{2}$ for $V\neq\{0\}$.}
\end{definition}

Similar angles \cite{Gluck1967,Gunawan2005,Jiang1996} that project from the smaller to the larger subspace do not satisfy a triangle inequality.
We project from $V$ to $W$ even if $\dim V>\dim W$, in which case $\Theta_{V,W}=\frac{\pi}{2}$.
As a result, in general $\Theta_{V,W}\neq \Theta_{W,V}$ if dimensions are different. 
This reflects the dimensional asymmetry of subspaces, and lets the angle carry some dimensional data ($\Theta_{V,W}\neq\frac{\pi}{2} \Rightarrow \dim P_W(V) = \dim V \leq \dim W$),
what simplifies proofs (\eg in \Cref{pr:full Grassmannian}).
Many of our results, like the triangle inequality, only hold in full generality thanks to the angle asymmetry. 
	\SELF{\ref{pr:spherical Pythagorean theorem}, \ref{pr:Theta},\ref{pr:products Theta}, 
	\ref{pr:formula any base dimension}, \ref{pr:formula complementary angle bases}}

By \Cref{pr:relative position}, having equal asymmetric angles does not mean pairs of subspaces can be related via $U(X)$, so $\Theta_{V,W}$ does not describe completely the relative position of $V$ and $W$.
What it does is codify information about projection factors \cite{Mandolesi_Pythagorean}.

\begin{definition}
	Let $V,W \in G(X)$, $k=\dim V_\R$ and $|\cdot|_k$ be the $k$-dimensional Lebesgue measure in $X_\R$. 
	The \emph{projection factor} of $V$ on $W$ is $\pi_{V,W}=\frac{|P_W(S)|_k}{|S|_k}$, for a Lebesgue measurable set $S\subset V$ with $|S|_k\neq 0$.
\end{definition}

\begin{proposition}\label{pr:projection factor Theta}
	$\pi_{V,W} = \begin{cases}
		\cos\Theta_{V,W} &\text{if $\F=\R$}, \\
		\cos^2\Theta_{V,W} &\text{if $\F=\C$.}
	\end{cases}$
\end{proposition}
\begin{proof}
	Immediate, as $\|A\|$ and $\|P_W A\|$ (squared, if $\F=\C$) are $k$-volumes of a parallelotope and its projection.
\end{proof}

So $\cos\Theta_{V,W}$ (squared, if $\F=\C$) measures the contraction of volumes orthogonally projected from $V$ to $W$ (Fig.\,\ref{fig:projecao}).

\begin{figure}
	\centering
	\includegraphics[width=0.6\linewidth]{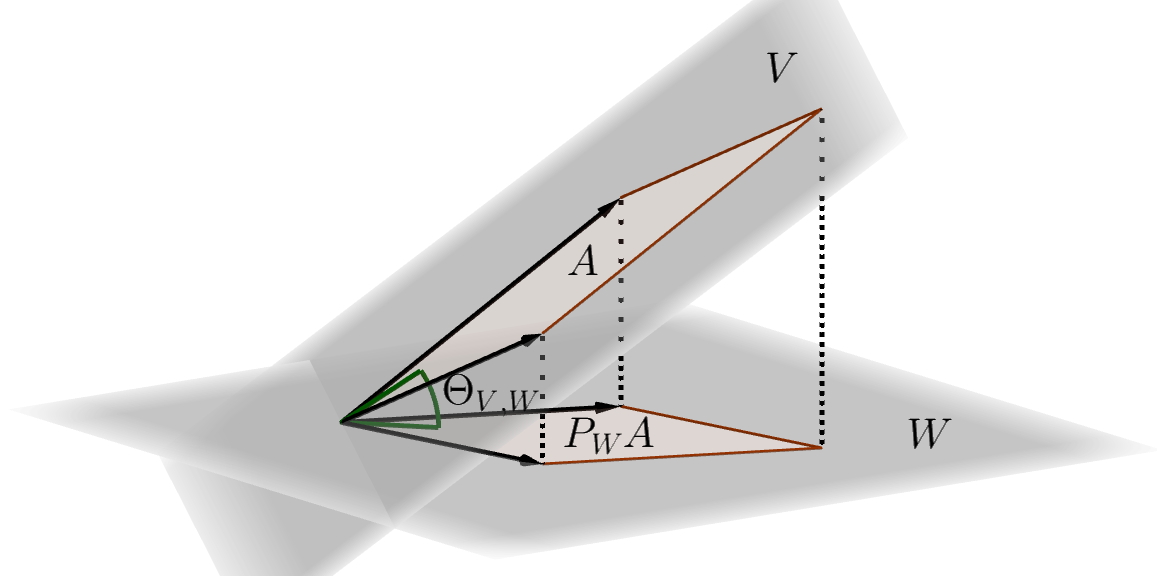}
	\caption{The area of $A$ contracts by $\cos\Theta_{V,W}$ when projected on $W$}
	\label{fig:projecao}
\end{figure}

\begin{corollary}\label{pr:Theta R C}
	$\cos\Theta_{V_\R,W_\R} = \cos^2 \Theta_{V,W}$, if $\F=\C$.
\end{corollary}
\begin{proof}
	$\pi_{V,W} = \pi_{V_\R,W_\R}$, as the Lebesgue measure is taken in $X_\R$.
\end{proof}

So, in general, $\Theta_{V_\R,W_\R} > \Theta_{V,W}$ if $\F=\C$, what may seem strange since $X_\R \cong X$ as metric spaces. 
An explanation is that these angles are different ways to encode the same projection factor.
One might say $\Theta_{V,W}$ should be defined as equal to $\Theta_{V_\R,W_\R}$, but formulas would differ if $\F=\C$: \eg $|\inner{A,B}| = \|A\| \|B\|\sqrt{\cos \Theta_{V_\R,W_\R}}$ in \Cref{pr:products Theta}.
Another inconvenient is that working in $X_\R$ increases dimensions and wastes  symmetries of the complex structure, leading to repeated principal angles (see the observations and examples after \Cref{pr:Pythagorean identity}).
Quantum theory gives us another reason to prefer $\Theta_{V,W}$:
the Bures angle \cite{Bengtsson2017} for pure quantum states $\psi$ and $\phi$ corresponds to $\Theta_{\C\psi,\C\phi}$.

\begin{proposition}\label{pr:Theta prod cos}
	For nonzero $V\in G_p(X)$ and $W\in G_q(X)$,
	with principal angles $\theta_1,\ldots,\theta_{\min\{p,q\}}$,
	\begin{equation}\label{eq:Theta prod cos}
		\cos\Theta_{V,W} = 
		\begin{cases}
			\prod_{i=1}^p \cos\theta_i &\text{if } p\leq q, \\
			0 &\text{otherwise.}
		\end{cases}
	\end{equation}
\end{proposition}
\begin{proof}
	Follows from \Cref{pr:P ei}\ref{it:P e1...ep}.	
\end{proof}

In general $\Theta_{V,W} > \theta_i \ \forall i$, so no line of $V$ makes such angle with $W$.
	\SELF{$\Theta_{V,W}\geq \theta_m$, equality iff $\theta_m=\frac{\pi}{2}$, or $\dim V\leq\dim W$ and $\theta_i=0$ for all $i<m$}
This formula gives another way to look at \Cref{pr:projection factor Theta}:
if $\F=\R$, each principal axis $\R e_i\subset V$ projected to $W$ contracts by $\cos\theta_i$, so $p$-volumes contract by $\cos\Theta_{V,W}$;
if $\F=\C$, each $\cos\theta_i$ describes the contraction of 2 real axes, $\R e_i$ and $\R(\im e_i)$, so $2p$-volumes in $V$ contract by $\cos^2\Theta_{V,W}$.

\begin{example}\label{ex:Theta > theta}
	In \Cref{ex:real principal angles}, all lines in $V$ make a $45^\circ$ angle with $W$, but $\Theta_{V,W}= \cos^{-1}(\frac{\sqrt{2}}{2}\cdot\frac{\sqrt{2}}{2}) = 60^\circ$, so that, when projected from $V$ to $W$, lengths contract by $\frac{\sqrt{2}}{2}$, areas by $\frac12$.
	As $\dim W > \dim V$, volumes vanish when projected from $W$ to $V$, and $\Theta_{W,V} = 90^\circ$.
\end{example}

\begin{example}\label{ex:complex asymmetric angle}
	In \Cref{ex:complex principal angles},
	$\Theta_{V,W}=\cos^{-1}(\frac{\sqrt{2}}{2}\cdot\frac{1}{2})\cong 69.3^\circ$, while $\Theta_{V_\R,W_\R}=\cos^{-1}(\frac{\sqrt{2}}{2}\cdot\frac{\sqrt{2}}{2}\cdot\frac{1}{2}\cdot\frac{1}{2}) \cong 82.8^\circ$. 
	Both angles convey the same information, that $4$-volumes in $V$ contract by a factor $\cos^2\Theta_{V,W}=\cos \Theta_{V_\R,W_\R}=\frac{1}{8}$ when projected on $W$.
\end{example}

For equal dimensions, $\Theta_{V,W}$ is symmetric, equals the Fubini-Study metric (so it is an angle between lines in $\bigwedge X$), and is related to the Binet-Cauchy and chordal-$\wedge$ metrics:

\begin{corollary}\label{pr:Theta equal dim}
	If $\dim V = \dim W$ then $\Theta_{V,W} = \Theta_{W,V} = d_{FS}(V,W)$, $d_{BC}(V,W) = \sin \Theta_{V,W}$ and $d_{c\wedge}(V,W) = 2\sin\frac{\Theta_{V,W}}{2}$.
\end{corollary}

If $p = \dim V \neq \dim W$, $\Theta_{V,W}$ is an angle between a line $\LL = \bigwedge^p V$ and a subspace $\WW = \bigwedge^p W$
(given, if $\LL=\Span\{A\}$, by $\theta_{\LL,\WW} = \theta_{A,P_{\WW}A}$). 
	\OMIT{$ = \cos^{-1} \frac{\|P_{\WW} A\|}{\|A\|}$}

\begin{proposition}\label{pr:angle external powers}
	$\Theta_{V,W}= \theta_{\bigwedge^p V,\bigwedge^p W}$ for $V,W\in G(X)$ and $p=\dim V$.
\end{proposition}
\begin{proof}
	For $0\neq A \in \bigwedge^p V$, we have $P_{\bigwedge^p W} A = P_W A$, so that 
	$\Theta_{V,W} = \cos^{-1} \frac{\|P_W A\|}{\|A\|} = \theta_{A,P_W A} = \theta_{\bigwedge^p V,\bigwedge^p W}$.
\end{proof}

If $p > \dim W$ then $\WW=\{0\}$ and $\theta_{\LL,\WW} = \frac\pi2$, so this result reflects the angle asymmetry.
It gives another reason why $\Theta_{V_\R,W_\R} \neq \Theta_{V,W}$ if $\F=\C$:
these are angles in exterior algebras, and $\bigwedge(X_\R) \not\cong \bigwedge X$.

\begin{proposition}\label{pr:Theta}
	Let $V,W\in G(X)$.
	\begin{enumerate}[i)]
		\item $\Theta_{V,W}=0 \Leftrightarrow V\subset W$, and $\Theta_{V,W}=\frac \pi 2 \Leftrightarrow V \pperp W$. \label{it:Theta 0 pi2}
		\item $\Theta_{V,W}=\Theta_{V,P_W(V)}$.\label{it:Theta P(V)}
		\item $\Theta_{V,W}=\Theta_{V,W\oplus U}$ for any subspace $U\perp (V+W)$.\label{it:Theta W+U} 
		\item $\Theta_{V,W}=\Theta_{V',W'}$ for $V' = (V\cap W)^\perp \cap V$ and $W' = (V\cap W)^\perp \cap W$. \label{it:orth complem inter} 
		\item $\Theta_{T(V),T(W)} = \Theta_{V,W}$ for any $T\in U(X)$. \label{it:transformation}
		\item $\Theta_{V^\perp,W^\perp} = \Theta_{W,V}$. \label{it:Theta perp perp}
	\end{enumerate}
\end{proposition}
\begin{proof}
	Follow from \Cref{pr:Theta prod cos} combined with:
	\emph{(\ref{it:Theta 0 pi2})} \Cref{pr:pperp}\ref{it:pperp q<p ou pi2};
	\emph{(\ref{it:Theta P(V)})} \ref{it:Theta 0 pi2} and \Cref{pr:P ei}\ref{it:PV};
	\emph{(\ref{it:transformation})} \Cref{pr:relative position};
	\emph{(\ref{it:Theta perp perp})} \Cref{pr:thetas perps} and $\dim V^\perp > \dim W^\perp \Leftrightarrow \dim W > \dim V$.
\end{proof}

These properties rely on the angle asymmetry (\eg take \ref{it:Theta P(V)} with perpendicular planes in $\R^3$),
and \ref{it:Theta 0 pi2} sheds light on it: $\Theta_{V,W}$ shows how far $V$ is from being contained in $W$, and is $\frac \pi 2$  if $V$ has a line orthogonal to $W$. 
A line $L$ can go from being contained to being orthogonal to a plane $W$, so $\Theta_{L,W}$ can have any value;  $W$ is never any closer to being contained in $L$, and always has a line orthogonal to $L$, so $\Theta_{W,L}$ is always $\frac{\pi}{2}$.

\begin{proposition}\label{pr:determinant projection}
$\cos^2\Theta_{V,W}=\det(\bar{\mathbf{P}}^T \mathbf{P})$, where $\mathbf{P}$ is a matrix for the orthogonal projection $V\rightarrow W$ in orthonormal bases of $V$ and $W$.
If $\dim V=\dim W$ then $\cos\Theta_{V,W}=|\det \mathbf{P}|$.
\end{proposition}
\begin{proof}
	Follows from \eqref{eq:Theta prod cos}, as in associated principal bases $\mathbf{P}$ is a $q\times p$ diagonal matrix with the $\cos\theta_i$'s.
		\OMIT{If $p>q$ the diagonal of $\bar{\mathbf{P}}^T \mathbf{P}$ has $0$'s}
\end{proof}

\begin{figure}
	\centering
	\begin{subfigure}[b]{0.5\textwidth}
		\includegraphics[width=\textwidth]{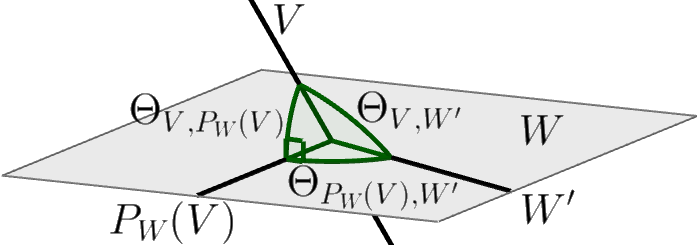}
		\caption{$\cos\Theta_{V,W'} = \cos\Theta_{V,P_W(V)}  \cos\Theta_{P_W(V),W'}$}
		\label{fig:spherical Pythagorean}
	\end{subfigure}
	\qquad
	\begin{subfigure}[b]{0.43\textwidth}
		\includegraphics[width=\textwidth]{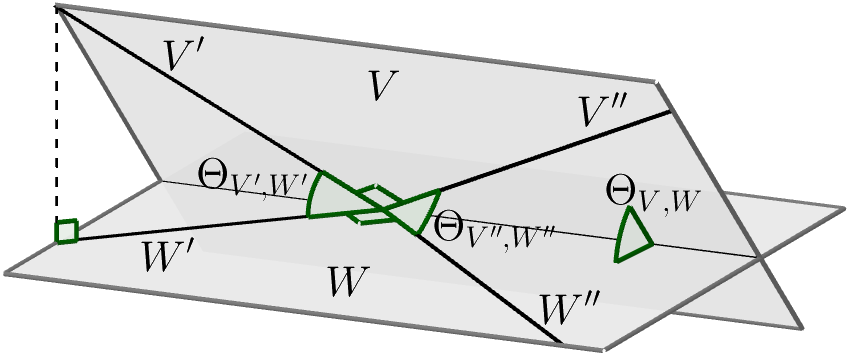}
		\caption{$\cos \Theta_{V,W} = \cos \Theta_{V',W'} \cos \Theta_{V'',W''}$}
		\label{fig:partition}
	\end{subfigure}
	\caption{Generalized spherical Pythagorean theorem, and angles for orthogonal partitions (the figures use lines and planes, but the formulas hold for any dimensions)}
	\label{fig:3 cos}
\end{figure}

\begin{proposition}\label{pr:spherical Pythagorean theorem}
$\cos\Theta_{V,W'} = \cos\Theta_{V,P_W(V)} \cos\Theta_{P_W(V),W'}$ for any $V,W,W'\in G(X)$ with $W'\subset W$.
\end{proposition}
\begin{proof}
\Cref{pr:Theta}\ref{it:Theta 0 pi2} lets us assume $V\not\pperp W$, so $\dim P_W(V)=\dim V$. 
	\OMIT{\ref{pr:pperp}}
Let $\mathbf{P}_1$, $\mathbf{P}_2$ and $\mathbf{P}_3$ be matrices for the orthogonal projections $V \rightarrow W'$, $V \rightarrow P_W(V)$ and $P_W(V) \rightarrow W'$, respectively, in orthonormal bases.
Since $\mathbf{P}_1 = \mathbf{P}_3 \mathbf{P}_2$
	\OMIT{$v=P_Wv+b$, $b\perp W$, so $P_{W'}^V v = P_{W'} v = P_{W'} P_W v+0 = P_{W'}^{PV} P_{PV}^V v$}
and $\mathbf{P}_2$ is square, 
$\det(\bar{\mathbf{P}}_1^T \mathbf{P}_1) = |\det \mathbf{P}_2|^2 \cdot \det(\bar{\mathbf{P}}_3^T \mathbf{P}_3)$.
	\OMIT{$= \det(\bar{P}_2^T\bar{P}_3^TP_3P_2) = \det \bar{P}_2^T \det(\bar{P}_3^TP_3)\det P_2$}
The result follows from \Cref{pr:determinant projection}.
\end{proof}

This formula generalizes the spherical Pythagorean theorem.
It relies on the angle asymmetry (\eg let $U=V\cap W$ for planes $V,W\subset \R^3$), as does the next result (\eg partition $V=\R^3$ into a line $V'$ and a plane $V''$, and let $W$ be another plane).
Fig.\,\ref{fig:3 cos} illustrates them.

\begin{proposition}\label{pr:Theta orthog partition}
	$\cos \Theta_{V,W} = \cos \Theta_{V',W'} \cos \Theta_{V'',W''}$ 
	for orthogonal partitions $V = V'\oplus V''$ and $W=W'\oplus W''$ with $W' = P_W(V')$.
\end{proposition}
\begin{proof}
	Given unit blades $A$ and $B$ with $[A] = V'$ and $[B] = V''$, $A \wedge B$ is a unit blade with $[A\wedge B] = V$.
	If $V' \pperp W$ then $\Theta_{V,W} = \Theta_{V',W'} = \frac\pi2$.
		\OMIT{\ref{pr:Theta}\ref{it:Theta 0 pi2}}
	If $V' \not\pperp W$ then $W'' = [P_W A]^\perp \cap W$, 
		\OMIT{\ref{pr:P ei}\ref{it:[P_W A] = P_W[A]}, \ref{pr:pperp}}
	so $P_W A \wedge P_W B = P_W A \wedge P_{W''} B$
		\OMIT{but $P_W B \neq P_U B$}
	and $\cos\Theta_{V,W} = \|P_W(A\wedge B)\| = \|P_W A \wedge P_W B\| = \|P_W A \wedge P_{W''} B\| = \|P_W A\| \|P_{W''} B\| = \cos\Theta_{V',W'} \cos\Theta_{V'',W''}$.
		\OMIT{then use \ref{pr:Theta}\ref{it:Theta P(V)}}
\end{proof}

\begin{proposition}\label{pr:Theta min max Theta subspaces}
	Let $V,V',W,W' \in G(X)$ with $V'\subset V$ and $W'\subset W$.
	\begin{enumerate}[i)]
		\item $\Theta_{V,W'} \geq \Theta_{V,W}$, with equality if and only if $V\pperp W$ or $P_W(V)\subset W'$.\label{it:W' sub W}
		\item $\Theta_{V',W} \leq \Theta_{V,W}$, with equality if and only if $V'\pperp W$ or $V'^\perp \cap V \subset W$.\label{it:V' sub V}		
	\end{enumerate}
\end{proposition}
\begin{proof}
	\emph{(\ref{it:W' sub W})} By Propositions \ref{pr:Theta}\ref{it:Theta P(V)} and \ref{pr:spherical Pythagorean theorem},
	$\cos\Theta_{V,W'} \leq \cos\Theta_{V,W}$, with equality if, and only if, $\Theta_{V,W} = \frac\pi2$ or $\Theta_{P_W(V),W'} = 0$.
		\OMIT{\ref{pr:Theta}\ref{it:Theta 0 pi2} leva ao enunciado}
	\emph{(\ref{it:V' sub V})} By \Cref{pr:Theta orthog partition},
	$\cos\Theta_{V,W} \leq \cos\Theta_{V',P_W(V')}$, 
		\OMIT{then use \ref{pr:Theta}\ref{it:Theta P(V)}}
	with equality if, and only if, $V'\pperp W$ 
		\OMIT{\ref{pr:Theta}\ref{it:Theta 0 pi2}}
	or $V'^\perp \cap V \subset (P_W (V'))^\perp \cap W$ ($\Leftrightarrow V'^\perp \cap V \subset W$).
\end{proof}

This implies the minimum angle in certain sets of subspaces is $\Theta_{V,W}$.
The decomposition in \Cref{df:PO} leads to special subspaces attaining this minimum (Fig.\,\ref{fig:angle VWperp}):
if $p=\dim V \leq \dim W = q$ then $\dim W_P = p$ and $\dim V \oplus W_\perp = q$, so the minima are attained in $G_p(W)$ and $G_q(X)$.

\begin{corollary}\label{pr:extrema}
	For $V,W \in G(X)$, 
	$\min\{\Theta_{V,U}:U \in G(W)\} = \Theta_{V,W} = \Theta_{V,W_P}$ and $\min\{\Theta_{Y,W}:Y \in G(X), Y\supset V\} = \Theta_{V,W} = \Theta_{V\oplus W_\perp,W}$.
\end{corollary}
\begin{proof}
	If  $0<\dim V\leq \dim W$, it follows from Propositions \ref{pr:thetas WP Wperp}, \ref{pr:Theta prod cos} and \ref{pr:Theta min max Theta subspaces}. 
		\OMIT{\ref{it:V' sub V},\ref{it:W' sub W}}
	If $V=\{0\}$ all values are $0$, and if $\dim V>\dim W$ all are $\frac\pi2$.
\end{proof}

\begin{figure}
	\centering
	\includegraphics[width=0.55\linewidth]{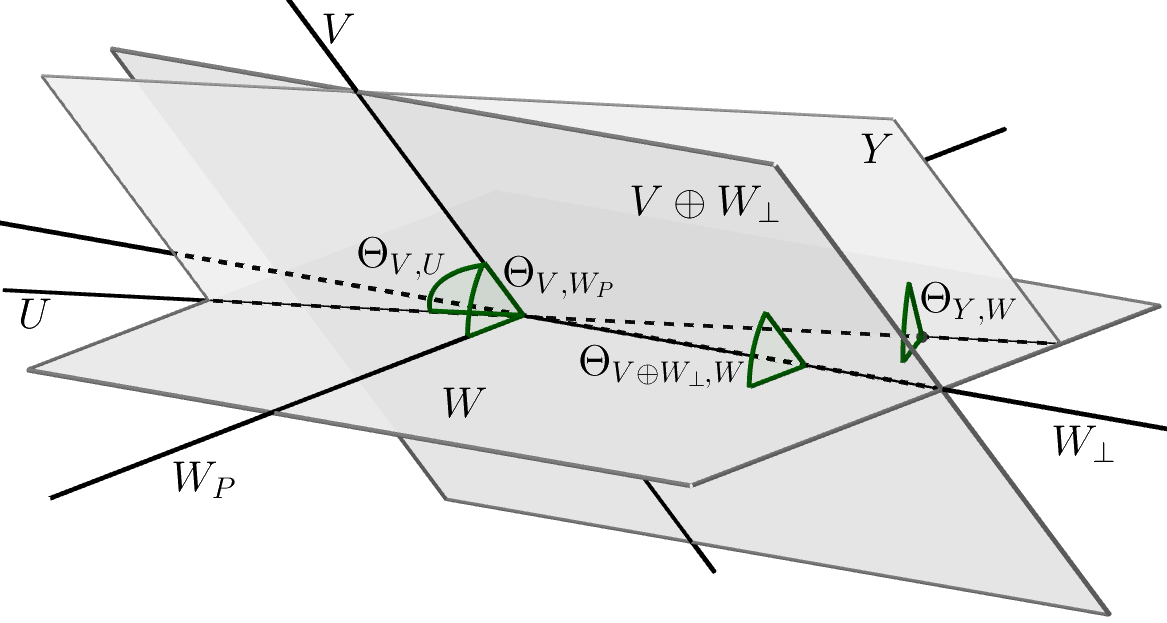}
	\caption{$\Theta_{V,W} = \Theta_{V,W_P} \leq \Theta_{V,U}$ if $U\subset W$,  $\Theta_{V,W} = \Theta_{V\oplus W_\perp,W} \leq \Theta_{Y,W}$ if $Y\supset V$}
	\label{fig:angle VWperp}
\end{figure}

\subsection{Metric properties}\label{sc:metric}

As seen, $\Theta_{V,W}$ gives the Fubini-Study metric on each $G_p(X)$.
We now prove that  on $G(X)$ it is an asymmetric metric, 
discuss its topologies, and obtain equality conditions for the oriented triangle inequality.

\begin{theorem}\label{pr:full Grassmannian}
	$G(X)$ is an asymmetric metric space, with distances given by $d(V,W) = \Theta_{V,W}$ for $V,W\in G(X)$.
\end{theorem}
\begin{proof}
	The first condition in \Cref{df:asymmetric metric} follows from \Cref{pr:Theta}\ref{it:Theta 0 pi2},
	and we must prove $\Theta_{U,W} \leq \Theta_{U,V} + \Theta_{V,W}$ for any $U,V,W \in G(X)$.
	We can assume $\Theta_{U,V},\Theta_{V,W}\neq\frac\pi 2$, and so $\Theta_{P_V(U),W}\neq \frac{\pi}{2}$ as well.
	Therefore $U$, $P_V(U)$ and $P_W P_V(U)$ have  the same dimension $p$, and
	\begin{align}
		\Theta_{U,W} &\leq \Theta_{U,P_W P_V(U)} \label{eq:triangle1}\\
		&\leq \Theta_{U,P_V(U)} + \Theta_{P_V(U),P_W P_V(U)} 
		= \Theta_{U,V} + \Theta_{P_V(U),W} \label{eq:triangle2}\\
		&\leq \Theta_{U,V} + \Theta_{V,W}, \label{eq:triangle3}
	\end{align}
	using Propositions \ref{pr:Theta min max Theta subspaces}\ref{it:W' sub W} in \eqref{eq:triangle1},  
	\ref{pr:spherical triangle inequality lines}, \ref{pr:angle external powers}, \ref{pr:Theta}\ref{it:Theta P(V)} in \eqref{eq:triangle2}, 
	and \ref{pr:Theta min max Theta subspaces}\ref{it:V' sub V} in \eqref{eq:triangle3}.
\end{proof}

\begin{figure}
	\centering
	\begin{subfigure}[b]{0.35\textwidth}
		\includegraphics[width=\textwidth]{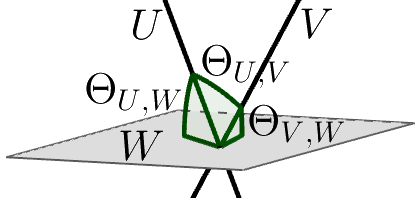}
		\caption{$\Theta_{U,W} \leq \Theta_{U,V} + \Theta_{V,W}$}
		\label{fig:oriented triangle inequality_1}
	\end{subfigure}
	\qquad
	\begin{subfigure}[b]{0.47\textwidth}
		\includegraphics[width=\textwidth]{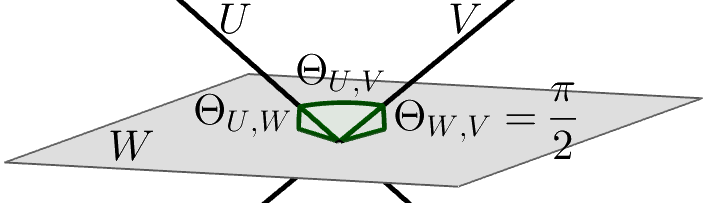}
		\caption{$\Theta_{U,V} \leq \Theta_{U,W} + \Theta_{W,V}$, even though $\Theta_{U,V} > \Theta_{U,W} + \Theta_{V,W}$}
	\label{fig:oriented triangle inequality_2}
\end{subfigure}
\caption{Oriented triangle inequalities for lines $U$ and $V$ and a plane $W$.}
\label{fig:oriented triangle inequality}
\end{figure}

Fig.\,\ref{fig:oriented triangle inequality} illustrates the oriented triangle inequality.
The angle asymmetry and the order of subspaces play a crucial role: in Fig.\,\ref{fig:oriented triangle inequality_2}, $\Theta_{W,V} = \frac{\pi}{2}$ and so $\Theta_{U,V} \leq \Theta_{U,W} + \Theta_{W,V}$, but $\Theta_{U,V} > \Theta_{U,W} + \Theta_{V,W}$.

The topologies induced by this asymmetric metric are natural for $G(X)$, as we argue in \Cref{sc:Other asymmetric metrics}, and reflect well the dimensional asymmetry of subspaces.
For example, a line $L$ and a plane $V$ have sensible neighborhoods in the backward topology $\tau^-$ of $G(\R^3)$:
for $0<r<\frac\pi2$, $B^-_r(L) = \{U\in G(X):\Theta_{U,L}<r\}$ has $\{0\}$ and lines inside a double cone around $L$ (Fig.\,\ref{fig:neighborhood line}), and $B^-_r(V)$ has $\{0\}$, lines and planes outside a double cone (Fig.\,\ref{fig:neighborhood plane}).
If $\Theta_{L,V}\rightarrow 0$, $L$ will end up in $B^-_r(V)$, 
and the constant $\Theta_{V,L} = \frac{\pi}{2}$ means $V$ is never any closer to being in $B^-_r(L)$.
Depending on the purpose, one can also use the forward topology $\tau^+$: $B^+_r(L) = \{U\in G(X):\Theta_{L,U}<r\}$ has $\R^3$, lines and planes intercepting the interior of the cone of Fig.\,\ref{fig:neighborhood line}, and $B^+_r(V)$ has $\R^3$ and planes outside the cone of Fig.\,\ref{fig:neighborhood plane}.
The symmetric topology $\tau$ is the same of the full Fubini-Study metric: $B_r(L)$ has only lines, and $B_r(V)$ only planes.
	\SELF{unless $r=\frac\pi2$}

While in $\tau$ the $G_p(X)$'s are disconnected from each other, in $\tau^\pm$ we have continuous paths linking subspaces of distinct dimensions: \eg for $U\subset W$, the paths
$V_-(t) = \begin{cases}
	U \text{ if } t < 0, \\
	W \text{ if } t \geq 0,
\end{cases}$
and
$V_+(t) = \begin{cases}
	U \text{ if } t \leq 0, \\
	W \text{ if } t > 0,
\end{cases}$
are continuous in $\tau^-$ and $\tau^+$, respectively, and so are the reversed paths.
\SELF{in $\tau^-/\tau^+$ strict ineq in smaller/larger subspace. \\ Revers. paths also contin.}
On the other hand, $\Theta_{U,W}$ is discontinuous in $\tau^\pm$ when dimensions change:
\eg if $U \subsetneq W$ then $U \in B^-_r(W)$ and $W \in B^+_r(U)$ for all $r$, but $\Theta_{U,U} = \Theta_{W,W} = 0$ and $\Theta_{W,U}=\frac\pi2$.
\SELF{in both topologies $\Theta_{U,W}$ is discontinuous in $U$ and $W$}

\begin{figure}
	\centering
	\begin{subfigure}[b]{0.45\textwidth}
		\includegraphics[width=\textwidth]{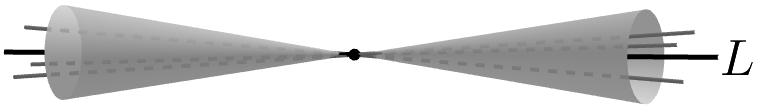}
		\caption{$B^-_r(L)$ has $\{0\}$ and all lines in the interior of a double cone}
		\label{fig:neighborhood line}
	\end{subfigure}
	\qquad
	\begin{subfigure}[b]{0.45\textwidth}
		\includegraphics[width=\textwidth]{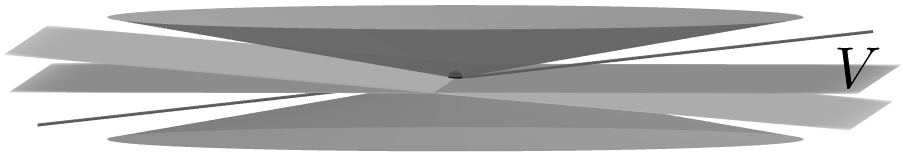}
		\caption{$B^-_r(V)$ has $\{0\}$ and all lines and planes in the exterior of a double cone}
		\label{fig:neighborhood plane}
	\end{subfigure}
	\caption{Dimensional asymmetry between backward balls of a line $L$ and a plane $V$}
	\label{fig:neighborhoods}
\end{figure}

The next lemma will help us obtain equality conditions for the triangle inequality.
As argued in \cite[p.\,519]{Qiu2005} it is important to analyze such extremal cases, which occur for example in geodesics whose lengths equal the distance between its endpoints.

\begin{lemma}\label{pr:linear combination blades}
	Let $A,B,C$ be nonzero $p$-blades with $[A], [B], [C]$ all distinct.
	If $B=\kappa A+\lambda C$ for $\kappa,\lambda \in \C$, there are $u,v,w\in X$ and a unit blade $D\in\bigwedge^{p-1} X$ with $v=\kappa u+\lambda w$, $A=u\wedge D$, $B=v\wedge D$, $C=w\wedge D$ and $[D] = [A]\cap [B]\cap [C]$. 
	We can choose $u$, $v$, $w$ in any complement of $[D]$, and if they are in $[D]^\perp$ then $\inner{u,w}=\inner{A,C}$.
\end{lemma}
\begin{proof}
	If $x\in [A]\cap [C]$ then $x\wedge B=x\wedge(\kappa A+\lambda C)=0$, so $x\in [B]$. 
	As the spaces are distinct, $\kappa,\lambda \neq 0$, so $A$ and $C$ are also linear combinations of the other blades.
	Thus $[A]\cap [C]=[A]\cap [B]=[B]\cap [C] = [A]\cap [B]\cap [C] = [D]$ for a unit $q$-blade $D$.
	Given a complement $X'$ of $[D]$, we have $A = A' \wedge D$, $B = B' \wedge D$ and $C = C' \wedge D$ for blades $A', B', C' \in \bigwedge^{p-q} X'$ with $[A']$, $[B']$ and $[C']$ all disjoint.
	As $(B'-\kappa A'-\lambda C')\wedge D=0$
		\OMIT{$B-\kappa A-\lambda C=0$}
	and $B'-\kappa A'-\lambda C'\in\bigwedge X'$, \Cref{pr:wedge disjoint multivectors}
		\OMIT{Precisa pois não sabe se \\ $B'-\kappa A'-\lambda C'$ é blade.}
	gives $B'=\kappa A'+\lambda C'$.
	So for any nonzero $u'\in [A']$ and $w'\in [C']$ we have $B'\wedge u'\wedge w' = 0$, and as	$B'\wedge u'\neq 0$ (by \Cref{pr:wedge disjoint multivectors})
		\OMIT{$[A']$ and $[B']$ are disjoint}
	this means $w'\in [B'\wedge u'] = [B']\oplus [u']$. 
	Since $w'\not\in [B']$ and $u'$ was chosen at will in $[A']$, this implies $\dim [A']=1$.
		\OMIT{If $w=b_1+\lambda_1 a_1=b_2+\lambda_2a_2$ for $a_1,a_2\in [A']$, $b_1,b_2\in [B']$ and $\lambda_1,\lambda_2\neq 0$, then $\lambda_1a_1-\lambda_2a_2=b_2-b_1\in [A']\cap [B']$, so $\lambda_1a_1=\lambda_2a_2$.}
	Thus $A',B',C'$ are vectors $u,v,w$.
	If $X'=[D]^\perp$ then $\inner{A,C} = \inner{u\wedge D,w\wedge D} = \inner{u,w}\cdot\|D\|^2$.
\end{proof}

\begin{proposition}\label{pr:triangle equality}
	$\Theta_{U,W} = \Theta_{U,V} + \Theta_{V,W}$ for $U,V,W \in G(X)$ in the following cases, and no other:
		\SELF{(\!\emph{\ref{it:VW}}) and (\!\emph{\ref{it:UV}}) intersect if $U\subset V \subset W$}
	\begin{enumerate}[i)]
		\item $V\subset W$, and $U\pperp W$ or $P_W(U)\subset V$; \label{it:VW} 	
		\item $U\subset V$, and $U\pperp W$ or $U^\perp\cap V\subset W$; \label{it:UV}
		\item $U = [u]\oplus R$, $V = [v]\oplus S$ and $W = [w]\oplus T$ for subspaces $R\subset S\subset T$ and aligned	$u,v,w\in T^\perp$ with $v=\kappa u+\lambda w$ for $\kappa, \lambda >0$.\label{it:UVW ABC}
			\SELF{Com $\geq$ teria interseção com (\!\emph{\ref{it:VW}}) and (\!\emph{\ref{it:UV}}), sem englobar nenhuma parte delas}
	\end{enumerate}
\end{proposition}
\begin{proof}
	By \Cref{pr:Theta min max Theta subspaces}, 
		\OMIT{\ref{pr:Theta}\ref{it:Theta 0 pi2}, e \ref{pr:Theta min max Theta subspaces}\ref{it:V' sub V} p/i, \ref{pr:Theta min max Theta subspaces}\ref{it:W' sub W} p/ii}
	(\!\emph{\ref{it:VW}}) corresponds to $\Theta_{V,W}=0$ and $\Theta_{U,W} = \Theta_{U,V}$, 
	and (\!\emph{\ref{it:UV}}) to $\Theta_{U,V}=0$ and $\Theta_{U,W} = \Theta_{V,W}$.
	(\!\emph{\ref{it:UVW ABC}}) has 
		\OMIT{\ref{pr:Theta}(\ref{it:Theta W+U},\ref{it:orth complem inter}) }
	$\Theta_{U,W}=\theta_{[u],[w]}$,
		\OMIT{as $u$ and $w$ are aligned}
	$\Theta_{U,V} = \theta_{[u],[v]}$, $\Theta_{V,W} = \theta_{[v],[w]}$,
	and \Cref{pr:spherical triangle inequality lines} gives the equality.

	Suppose $\Theta_{U,W} = \Theta_{U,V} + \Theta_{V,W}$ but (\!\emph{\ref{it:VW}}) and (\!\emph{\ref{it:UV}}) do not hold (and so $\Theta_{U,V},\Theta_{V,W}\neq\frac\pi 2$).
	Equalities in \eqref{eq:triangle1} and \eqref{eq:triangle3} give 
	$\Theta_{U,P_W P_V(U)} = \Theta_{U,W}$ and $\Theta_{P_V(U),W} = \Theta_{V,W}$.
		\OMIT{$= \Theta_{P_V(U),P_W P_V(U)}$}
	These angles and $\Theta_{U,P_V(U)} = \Theta_{U,V}$ are nonzero,
	so $U$, $P_V(U)$ and $P_W P_V(U)$ are distinct $p$-subspaces.
	By \Cref{pr:angle external powers}, equality in \eqref{eq:triangle2} means $\theta_{J,L} = \theta_{J,K} + \theta_{K,L}$ for  $J = \bigwedge^p U$, $K = \bigwedge^p P_V(U)$ and $L = \bigwedge^p P_W P_V(U)$.
	\Cref{pr:spherical triangle inequality lines} gives aligned $p$-blades $A,B,C$ with $[A] = U$, $[B] = P_V U$, $[C] = P_W P_V(U)$ and $B = \kappa A + \lambda C$ for $\kappa,\lambda > 0$ (strict as the spaces are distinct).
	\Cref{pr:linear combination blades} gives $u,v,w \in X$ orthogonal to $R = [A]\cap [B]\cap [C]$, 
	with   
	$U = [u]\oplus R$, $P_V(U) = [v]\oplus R$, $P_W P_V(U) = [w]\oplus R$,
	$v=\kappa u +\lambda w$ and $\inner{u,w}=\inner{A,C} \geq 0$, so $u,v,w$ are aligned. 
	
	Let $V = [v]\oplus R \oplus S'$ be an orthogonal partition.
	As $P_V u \in [v]\oplus R$, we have $u \perp S'$ and $w = \frac{v - \kappa u}{\lambda} \perp S'$.
	And as $\Theta_{P_V(U),W} = \Theta_{V,W} \neq \frac{\pi}{2}$, 
		\OMIT{equality in \eqref{eq:triangle3}}
	\Cref{pr:Theta min max Theta subspaces}\ref{it:V' sub V} gives $S' = P_V(U)^\perp\cap V \subset W$.
	So we have another orthogonal partition $W = [w]\oplus R \oplus S' \oplus T'$.
	As $v\in P_V(U)$, we have $P_W v \in [w]\oplus R$, so $v \perp T'$ and $u = \frac{v-\lambda w}{\kappa} \perp T'$.
	All conditions of (\!\emph{\ref{it:UVW ABC}}) are satisfied with $S = R\oplus S'$ and $T=R \oplus S' \oplus T'$.
\end{proof}

For equal dimensions the equality conditions become simpler \cite{Jiang1996}:
\CITE{Jiang1996: Real case, $1<p<n$, all distinct}

\begin{corollary}\label{pr:triangle equality same dim}
	$\Theta_{U,W} = \Theta_{U,V} + \Theta_{V,W}$ for $U,V,W\in G_p(X)$ if, and only if, $V=U$ or $W$, or
	$U = [u]\oplus R$, $V = [v]\oplus R$ and $W = [w]\oplus R$ for $R \in G_{p-1}(X)$ and aligned	$u,v,w\in R^\perp$ with $v=\kappa u+\lambda w$ for $\kappa, \lambda > 0$.
\end{corollary}

\subsection{Other properties}\label{sc:properties}

The asymmetric angle is linked to various products of Grassmann and Clifford algebras \cite{Dorst2007,Mandolesi_Products}, which we use to obtain more properties.

The \emph{contraction} or \emph{interior product} $A \lcontr B$ of $A,B\in\bigwedge X$ is defined by
$\inner{C,A \lcontr B} = \inner{A\wedge C,B}$ for any $C\in\bigwedge X$ (there are other conventions \cite{Browne2012,Dorst2007,Mandolesi_Contractions}).
If $\F=\C$, it is conjugate linear on $A$.
If $A\in\bigwedge^p X$ and $B\in\bigwedge^q X$ then $A\lcontr B \in\bigwedge^{q-p} X$.
It is asymmetric, with $A\lcontr B = 0$ if $p>q$.
If $p=q$ then $A\lcontr B = \inner{A,B}$.
If $p\leq q$ and $B= v_1\wedge\cdots\wedge v_q$, 
	\SELF{$\epsilon_{\ii\ii'} = (-1)^{\|\ii\|+\frac{p(p+1)}{2}}$, as ordering $\ii\ii'$ takes $(i_1-1)+\cdots+(i_p-p)$ transpositions}
\begin{equation}\label{eq:contraction coordinate decomposition} 
	A \lcontr B = \sum_{\ii\in\I_p^q} \epsilon_{\ii\ii'} \,\inner{A,B_\ii}\, B_{\ii'},
\end{equation}
with $B_\ii = v_\ii$  
for $\ii\in\II_p^q$
(so $B = \epsilon_{\ii\ii'} \, B_\ii\wedge B_{\ii'}$).
	\SELF{We do not require $B\neq 0$}
For $(v_1,\ldots,v_n)$ orthonormal and $\ii,\jj\in\II^n$, if $\ii\subset\jj$ then $v_{\ii} \lcontr v_{\jj} = \epsilon_{\ii(\jj-\ii)}  v_{\jj-\ii}$, otherwise $v_{\ii} \lcontr v_{\jj} = 0$.
For nonzero blades, $A\lcontr B=0$ if $[A]\pperp [B]$, otherwise $[A\lcontr B] = [A]^\perp \cap [B]$.

\begin{proposition}\label{pr:products Theta}
	$\|A\lcontr B\| = \|A\|\|B\| \cos \Theta_{[A],[B]}$ for blades $A\in \bigwedge^p X$ and $B\in \bigwedge^q X$. If $p=q$, $|\inner{A,B}| = \|A\|\|B\|\cos\Theta_{[A],[B]}$.
		\CITE{real case: Edelman1999 p.337}
\end{proposition}
\begin{proof}
	We can assume $0<p\leq q$, 
		\OMIT{\ref{pr:Theta}\ref{it:Theta 0 pi2}}
	$A = e_1 \wedge \cdots \wedge  e_p$ and $B = f_1 \wedge \cdots \wedge  f_q$ for associated principal bases $(e_1,\ldots,e_p)$ and $(f_1,\ldots,f_q)$ of $[A]$ and $[B]$.
	By \eqref{eq:contraction coordinate decomposition}, $\|A\lcontr B\| = \inner{e_{1\cdots p},f_{1\cdots p}} = \prod_{i=1}^p \cos\theta_i$, so it follows from \eqref{eq:Theta prod cos}.
\end{proof}

The angle asymmetry matches that of $A\lcontr B$, so the formula for $\|A\lcontr B\|$ holds even if $p>q$.
These formulas are an easy way to compute $\Theta_{V,W}$. 

\begin{example}\label{ex:contraction}
	Let $(u_1,\ldots,u_5)$ be the canonical basis of $\R^5$, $v_1=2u_1-u_2$, $v_2 = 2u_1+u_3$, $v_3=u_3$, $w_1=u_2+u_5$, $w_2=u_3-u_4$ and $w_3=u_4$.
	For $A = v_{12} = 2u_{12}+2u_{13}-u_{23}$, $B = w_{12} = u_{23} - u_{24} - u_{35} +u_{45}$, $C = w_{123} = u_{234}+u_{345}$ and $D=v_{123} = 2u_{123}$, 
	we find 
	$\|A\| = 3$, $\|B\| = 2$, $\|C\| = \sqrt{2}$,
	$\inner{A,B} = -1$, $A\lcontr C = -u_4$ and $\inner{C,D} = 0$,
	so $\Theta_{[A],[B]} = \cos^{-1}\frac{1}{6} \cong 80.4^\circ$, $\Theta_{[A],[C]} = \cos^{-1}\frac{1}{3\sqrt{2}} \cong 76.4^\circ$ and $\Theta_{[C],[D]} = 90^\circ$.
\end{example}

We can write the formulas using matrices.
The first one is asymmetric, but works for similar angles if $p\leq q$, and is simpler than another of \cite{Gunawan2005}.

\begin{proposition}\label{pr:formula any base dimension}
	Given bases $(v_1,\ldots,v_p)$ of $V$ and $(w_1,\ldots,w_q)$ of $W$, let 
	$\AA_{p \times p} =\big(\inner{v_i,v_j}\big)$, 
	$\BB_{q \times q} =\big(\inner{w_i,w_j}\big)$ and
	$\CC_{q \times p} =\big(\inner{w_i,v_j}\big)$. 
	Then $\cos^2 \Theta_{V,W} = \frac{\det(\bar{\CC}^T  \BB^{-1}\CC)}{\det \AA}$.
	If $p = q$ then $\cos^2 \Theta_{V,W} = \frac{|\det \CC\,|^2}{\det \AA \cdot\det \BB}$.
\end{proposition}
\begin{proof}
	Assume $p\leq q$, otherwise $\Theta_{V,W} = \frac \pi 2$ and $\det(\bar{\CC}^T \! \BB^{-1}\CC) = 0$, as it is a $p\times p$ matrix of rank at most $q$.
	Let $\CC_\ii$ and $\BB_{\ii'}$ be the $p\times p$ and $(q-p)\times q$ matrices formed by the lines of $\CC$ with indices in $\ii\in\I_p^q$, 
	and the lines of $\BB$ with indices not in $\ii$.
	For $A=v_1\wedge\cdots \wedge v_p$ and $B=w_1\wedge\cdots\wedge w_q$, Laplace expansion, Schur's identity 
		\CITE{Muir2003 p.80 \\ Brualdi1983}
	and \eqref{eq:contraction coordinate decomposition} give
	$\|A\lcontr B\|^2 = \inner{A\wedge(A\lcontr B),B} 
	= \sum\limits_{\ii\in\I_p^q} \epsilon_{\ii\ii'} \,\inner{B_\ii,A}\, \inner{A\wedge B_{\ii'},B}
	=\sum\limits_{\ii\in\I_p^q} \epsilon_{\ii\ii'} \, \det \CC_\ii \cdot \begin{vsmallmatrix}
		\bar{\CC}^T \\[1pt] \BB_{\ii'}
	\end{vsmallmatrix}
	= (-1)^{p(q-p)} \sum\limits_{\ii\in\I_p^q} \epsilon_{\ii\ii'} \, \det \CC_\ii \cdot \begin{vsmallmatrix}
		\BB_{\ii'} \\[1pt] \bar{\CC}^T
	\end{vsmallmatrix}
	= (-1)^{pq+p} \cdot \begin{vsmallmatrix}
		\CC & \BB \\[1pt] \mathbf{0} &  \bar{\CC}^T
	\end{vsmallmatrix} 
	= (-1)^p \cdot \begin{vsmallmatrix}
		\mathbf{0} &  \bar{\CC}^T \\[1pt] \CC & \BB 
	\end{vsmallmatrix} 
	= (-1)^p \det \BB \cdot \det(-\bar{\CC}^T \BB^{-1}\CC)
	= \det \BB \cdot \det(\bar{\CC}^T \BB^{-1}\CC)$.
	The result follows from \Cref{pr:products Theta}. 
\end{proof}

\begin{example}\label{ex:formula distinct dim}
	In $\R^4$, let $V=[v]$ and $W=[w_{12}]$ for $v=(1,0,1,0)$, $w_1=(0,1,1,0)$, $w_2=(1,2,2,-1)$.
	With $\AA=(2)$,
	$\BB=\begin{psmallmatrix}
		2 & 4 \\ 4 & 10
	\end{psmallmatrix}$,
	$\CC=\begin{psmallmatrix}
		1 \\ 3
	\end{psmallmatrix}$, 
	we find $\Theta_{V,W}= 45^\circ$, as one can verify by projecting $v$ on $W$. 
	Switching the roles of $V$ and $W$, we now have 
	$\AA=\begin{psmallmatrix}
		2 & 4 \\ 4 & 10
	\end{psmallmatrix}$,
	$\BB=(2)$,
	$\CC=(1 \ 3)$ and $\Theta_{W,V}= 90^\circ$, as expected since $\dim W>\dim V$.
\end{example}

\begin{example}\label{ex:formula bases}
	In $\C^3$, with $\xi=e^{\im\frac{2\pi}{3}}$, let  $V=[v_{12}]$ and $W=[w_{12}]$ for $v_1=(1,-\xi,0), v_2=(0,\xi,-\xi^2)$, $w_1=(1,0,0)$ and $w_2=(0,\xi,0)$. 
	With $\AA=\begin{psmallmatrix} 2 & -1 \\ -1 & 2 \end{psmallmatrix}$, $\BB= \begin{psmallmatrix} 1 & 0 \\ 0 & 1	\end{psmallmatrix}$, $\CC=\begin{psmallmatrix} 1 & 0 \\ -1 & 1	\end{psmallmatrix}$
	we find
	$\Theta_{V,W}=\cos^{-1}\frac{1}{\sqrt{3}} \cong 54.7^\circ$. 
	As $v=(\xi,\xi^2,-2)\in V$
		\OMIT{$v=\xi v_1 + 2\xi v_2$}
	and $w=(1,\xi,0)\in W$ are orthogonal to $V\cap W=[v_1]$, \Cref{pr:Theta}\ref{it:orth complem inter} gives
	$\Theta_{V,W}=\Theta_{[v],[w]}=\gamma_{v,w}$, and the Hermitian angle formula confirms the result.
\end{example}

\begin{proposition}\label{pr:Pythagorean identity}
	Let $(w_1,\ldots,w_n)$ be an orthogonal basis of $X$, and $V\in G_p(X)$.
	Then  $\sum_{\ii\in\I_p^n} \cos^2\Theta_{V,[w_\ii]} = 1$. 
\end{proposition}
\begin{proof}
	We can assume the basis is orthonormal,
	so $\sum_{\ii\in\I_p^n} |\inner{A,w_\ii}|^2 = \|A\|^2$ for $A\in\bigwedge^p V$.
	The result follows from \Cref{pr:products Theta}.
\end{proof}

If $\F=\C$ then $\dim X_\R = 2\dim X$, so there are more real coordinate subspaces $[w_\ii]$ than complex ones.
As angles in $X_\R$ are larger, this allows the above sum to hold for their smaller but more numerous cosines.

\begin{example}
	In $\C^2$, if $v=(\frac{\im}{2},\frac{\sqrt{3}}{2})$, $w_1=(1,0)$, $w_2=(0,1)$
	and $V=[v]$ then $\Theta_{V,[w_1]} = 60^\circ$ and $\Theta_{V,[w_2]} = 30^\circ$.
	In the underlying $\R^4$, $V_\R = [u_{12}]$ for $u_1 = v = (0,\frac12,\frac{\sqrt{3}}{2},0)$ and $u_2 = \im v = (-\frac12,0,0,\frac{\sqrt{3}}{2})$,
	and with the canonical basis $(y_1,\ldots,y_4) = (w_1,\im w_1, w_2, \im w_2)$ we obtain 
	$\Theta_{V_\R,[y_{12}]} = \cos^{-1}(\frac{1}{4}) \cong 75.5^\circ$,
	$\Theta_{V_\R,[y_{13}]} = \Theta_{V_\R,[y_{24}]} = \cos^{-1}(\frac{\sqrt{3}}{4}) \cong 64.3^\circ$,
	$\Theta_{V_\R,[y_{14}]} = \Theta_{V_\R,[y_{23}]} = 90^\circ$ and
	$\Theta_{V_\R,[y_{34}]} = \cos^{-1}(\frac{3}{4}) \cong 41.4^\circ$.
	In both cases the squared cosines add up to $1$.
\end{example}

\begin{example}
	In \Cref{ex:formula bases}, $(w_1,w_2,w_3)$ with $w_3=(0,0,\xi^2)$ is an orthonormal basis.  
	The unitary $T=\left(\begin{smallmatrix}
		0 & 0 & \xi \\ 
		\xi & 0 & 0 \\ 
		0 & \xi & 0
	\end{smallmatrix}\right)$ maps $[w_{12}]\mapsto [w_{23}]$, $[w_{23}]\mapsto [w_{13}]$ and preserves $V$, so $\Theta_{V,[w_{12}]} = \Theta_{V,[w_{23}]} = \Theta_{V,[w_{13}]}$
		\SELF{$T:w_1\mapsto w_2\mapsto w_3\mapsto w_1$, $Tv_1=v_2$,\\ $Tv_2= -(v_1+v_2)$}
	and \Cref{pr:Pythagorean identity} gives $\cos\Theta_{V,[w_{ij}]}=\frac{1}{\sqrt{3}}$.
	While $V$ is equally distant from all $[w_{ij}]$'s, in the underlying $\R^6$ this is not true for $V_\R$ and the 15 coordinate 4-spaces of the orthonormal basis $(w_1,\im w_1,w_2,\im w_2, w_3, \im w_3)$, or we would have $\cos \Theta_{V_\R,W_\R} = \frac{1}{\sqrt{15}}$,
	contradicting \Cref{pr:Theta R C}.
\end{example}

\Cref{pr:Pythagorean identity} leads \cite{Mandolesi_Pythagorean} to a known real volumetric Py\-thag\-o\-re\-an theorem 
\CITE{Conant1974} 
(the squared volume of a region in $V$ is the sum of squared volumes of its projections on all $[w_\ii]$'s), and a simpler complex one with non-squared volumes (by \Cref{pr:projection factor Theta}) and less $[w_\ii]$'s.
This is a prime example of the advantages of $\Theta_{V,W}$ over $\Theta_{V_\R,W_\R}$.
In \cite{Mandolesi_Born} we use this complex Pythagorean theorem, and an interpretation of quantum probabilities as projection factors 
(with \Cref{pr:Pythagorean identity} giving unit total probability), 
to propose a solution to the probability problem of Everettian quantum mechanics and show why the quantum space must be complex.

In general, the asymmetric angle with an orthogonal complement is not the usual complement, \ie $\Theta_{V,W^\perp} \neq \frac\pi2 - \Theta_{V,W}$ (\eg take a plane $V$ and a line $W$ in $\R^3$), and so $\sin \Theta_{V,W} \neq \cos\Theta_{V,W^\perp}$.
This sine involves projections not to $W^\perp$ but to certain coordinate subspaces partially orthogonal to $W$:

\begin{proposition}\label{pr:sine}
	For nonzero $V\in G_p(X)$ and $W\in G_q(X)$, extending a principal basis $(f_1,\ldots,f_q)$ of $W$ \wrt $V$ to an orthonormal basis $(f_1,\ldots,f_n)$ of $X$ we have
	$\sin^2 \Theta_{V,W} = \sum_{\ii\in\II_p^n, \ii\not\subset 1\cdots q} \cos^2 \Theta_{V,[f_\ii]}$. 
\end{proposition}
\begin{proof}
	Follows from \Cref{pr:Pythagorean identity}, as if $p>q$ then $\Theta_{V,W} = \frac\pi2$ and $\ii\not\subset 1\cdots q$ for all $\ii\in\II_p^n$,
	and if $p\leq q$ then 
		\OMIT{\ref{pr:Theta}\ref{it:Theta W+U}}
	$\sin^2 \Theta_{V,W} = 1 - \cos^2 \Theta_{V,W} = \left( \sum_{\ii\in\II_p^n} \cos^2 \Theta_{V,[f_\ii]} \right)  - \cos^2 \Theta_{V,[f_{1\cdots p}]} = \sum_{\ii\in\II_p^n, \ii\neq 1\cdots p} \cos^2 \Theta_{V,[f_\ii]}$
	and we have $\Theta_{V,[f_\ii]} = \frac\pi2$ if $\ii\subset 1\cdots q$ and $\ii\neq 1\cdots p$, as $f_i \perp V$ for $p<i\leq q$.
\end{proof}

\begin{figure}
	\centering
	\includegraphics[width=0.7\linewidth]{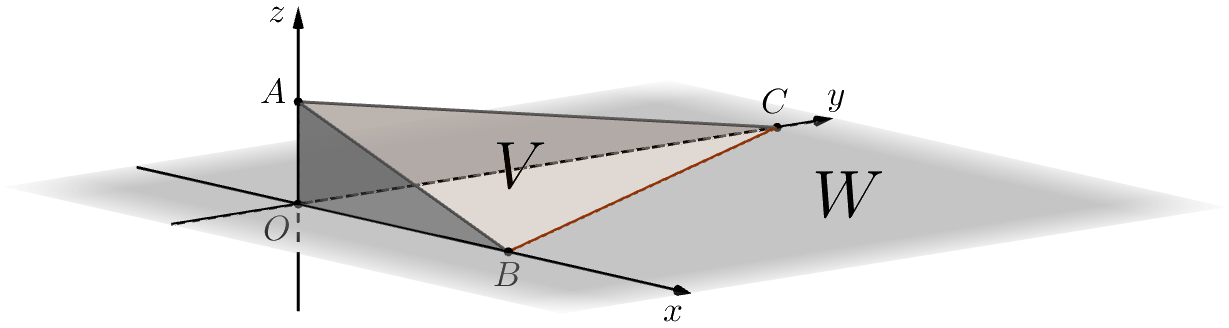}
	\caption{If $V$ is (parallel to) the plane of the triangle $ABC$, and $W$ is the $xy$ plane, then $\cos^2 \Theta_{V,W} = \frac{\area^2(OBC)}{\area^2(ABC)}$ and $\sin^2 \Theta_{V,W} = \frac{\area^2(OAB) + \area^2(OAC)}{\area^2(ABC)}$}
	\label{fig:projections-sin}
\end{figure}

With \Cref{pr:projection factor Theta}, this means that, given a unit volume in $V$, $\sin^2 \Theta_{V,W}$ is the sum of the squared (if $\F=\R$) volumes of its projections on all coordinate $p$-subspaces $[f_\ii]$ not contained in $W$ (Fig.\,\ref{fig:projections-sin}).
It gives a geometric interpretation for the Binet-Cauchy distance $d_{BC} = \sin \Theta_{V,W}$.
	\OMIT{\ref{pr:Theta equal dim}}

\begin{example}\label{ex:asymmetric angle complement}
	In \Cref{ex:real principal angles}, both principal angles of $V$ and $W^\perp=[f_{34}]$ are also $45^\circ$,
	so $\Theta_{V,W^\perp}=60^\circ = \Theta_{V,W}$.
	We also find $\Theta_{V,[f_\ii]} = 60^\circ$ for $\ii = 14$, $23$ or $34$, and $\Theta_{V,[f_\ii]} = 90^\circ$ for $\ii = 13$, $24$, $35$ or $45$, so that
	$\sum_{\ii\in\II_2^5, \ii\not\subset 125} \cos^2 \Theta_{V,[f_\ii]} = \frac{3}{4} = \sin^2 \Theta_{V,W}$.
\end{example}

\section{Related angles}\label{sc:Related angles}

Once a principal angle is $\frac\pi2$, $\Theta_{V,W}$ tells us nothing else about the others.
But we can gain some extra information from $\Theta_{V,W^\perp}$ and $\Theta_{V^\perp,W}$, which are, in general, somewhat independent of $\Theta_{V,W}$ \cite{Mandolesi_Grassmann}.

From $\Theta_{V,W^\perp}$ we form an angle which 
is similar to the max-correlation in some aspects, but gives finer information, taking all $\theta_i$'s into account.

\begin{definition}\label{df:Upsilon}
	$\Upsilon_{V,W}=\frac{\pi}{2}-\Theta_{V,W^\perp}$ is the \emph{disjointness angle} of $V,W\in G(X)$.
\end{definition}

The name is due to \ref{it:Upsilon 0 pi2} below: $\Upsilon_{V,W} = 0$ unless $V$ and $W$ are disjoint, so it gives a measure of how far $V$ and $W$ are from intersecting non-trivially, what is relevant in many applications.
The angle only reaches $\frac\pi2$ when $V\perp W$, so, unlike $\Theta_{V,W}$, it still gives more information when $V\pperp W$.

\begin{proposition}\label{pr:Upsilon}
	Let $V\in G_p(X)$ and $W\in G_q(X)$ be represented by blades $A$ and $B$, respectively.
	\begin{enumerate}[i)]
		\item $\Upsilon_{V,W}=0 \Leftrightarrow V\cap W\neq\{0\}$, and $\Upsilon_{V,W}=\frac\pi2 \Leftrightarrow V\perp W$. \label{it:Upsilon 0 pi2}
		\item $\sin \Upsilon_{V,W} = \frac{\|P_{W^\perp} A\|}{\|A\|}$. \label{it:Upsilon arcsin}
		\item $\Upsilon_{V,W} = \theta_{\bigwedge^p V, (\bigwedge^p(W^\perp))^\perp}$. \label{it:Upsilon angle exterior algebra}
		\item $\|A\wedge B\| = \|A\|\|B\| \sin \Upsilon_{[A],[B]}$. \label{it:exterior product}
		\item $\Upsilon_{V,W} = \Upsilon_{W,V}$. \label{it:symmetry complementary}
		\item If $V,W\neq\{0\}$ have principal angles $\theta_1,\ldots,\theta_m$ for $m=\min\{p,q\}$ then $\sin \Upsilon_{V,W} = \prod_{i=1}^m \sin\theta_i$.\label{it:prod sin}
		\item $\Upsilon_{V,W} = \Upsilon_{V,P_W(V)}$. \label{it:complem P(V)}
	\end{enumerate}
\end{proposition}
\begin{proof}
	\emph{(\ref{it:Upsilon 0 pi2}--\ref{it:Upsilon angle exterior algebra})} Immediate.
		\OMIT{\ref{pr:Theta}\ref{it:Theta 0 pi2} \\
			\ref{it:Upsilon angle exterior algebra}: \ref{pr:angle external powers}, $\Upsilon_{V,W} = \frac{\pi}{2} - \theta_{\bigwedge^p V, \bigwedge^p(W^\perp)} = \theta_{\bigwedge^p V, (\bigwedge^p(W^\perp))^\perp}$}
	\emph{(\ref{it:exterior product})} $\|A\wedge B\| = \|(P_{W^\perp} A)\wedge B\| = \|P_{W^\perp} A\| \|B\| = \|A\|\|B\|\sin \Upsilon_{V,W}$.
	\emph{(\ref{it:symmetry complementary})} Follows from \ref{it:exterior product}.
		\OMIT{or \ref{pr:Theta}\ref{it:Theta perp perp}}
	\emph{(\ref{it:prod sin})} Assuming $p\leq q$, 
	$A = e_1 \wedge \cdots \wedge  e_p$ and $B = f_1 \wedge \cdots \wedge  f_q$ for associated principal bases,
	we have $\sin \Upsilon_{V,W} = \| A\wedge  B\| = \|e_1\wedge f_1\|\cdots\|e_p\wedge f_p\| \cdot \|f_{p+1}\| \cdots \|f_q \|  = \sin\theta_1\cdots\sin\theta_p$.
	\emph{(\ref{it:complem P(V)})} Follows from \ref{it:Upsilon 0 pi2}, \ref{it:prod sin} and \Cref{pr:P ei}\ref{it:PV}.
\end{proof}

\begin{figure}
	\centering
	\includegraphics[width=0.7\linewidth]{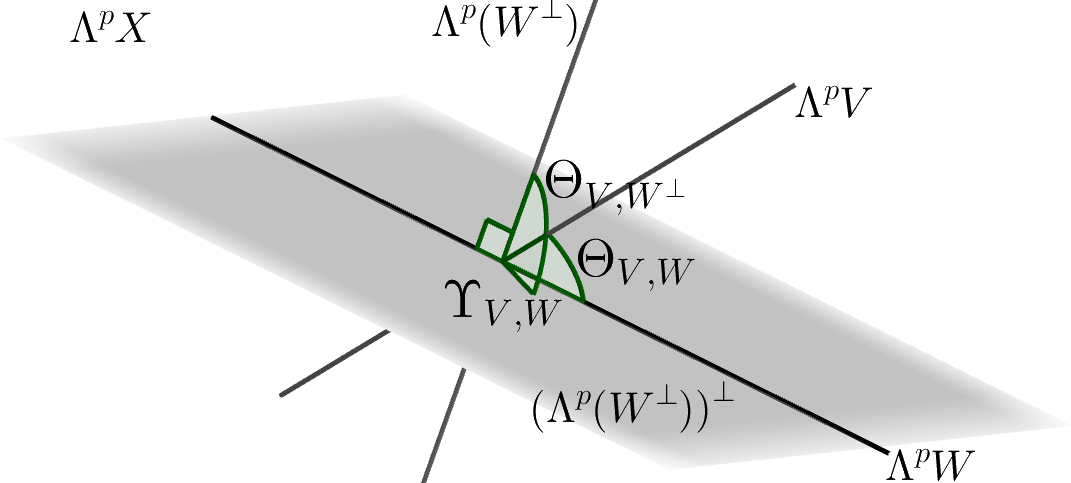}
	\caption{$\Theta_{V,W}$, $\Theta_{V,W^\perp}$ and $\Upsilon_{V,W}$ as angles in $\bigwedge^p X$, for $p = \dim V$, between a line $\bigwedge^p V$ and subspaces $\bigwedge^p W$, $\bigwedge^p (W^\perp)$ and $(\bigwedge^p (W^\perp))^\perp$ (shown as lines and a plane)}
	\label{fig:phi}
\end{figure}

By \ref{it:Upsilon 0 pi2}, $\Upsilon_{V,W}$ does not satisfy a triangle inequality (\eg take two lines and their plane).
By \ref{it:Upsilon arcsin}, $\sin \Upsilon_{V,W}$ (squared, if $\F=\C$) is the projection factor from $V$ to $W^\perp$.
Fig.\,\ref{fig:phi} illustrates \ref{it:Upsilon angle exterior algebra}.
The formula for $\|u \wedge v\|$ is generalized by \ref{it:exterior product}.
The unexpected symmetry in \ref{it:symmetry complementary} means $\Theta_{V,W^\perp} = \Theta_{W,V^\perp}$.
For $V,W\neq\{0\}$, \ref{it:prod sin} implies $\Upsilon_{V,W} \leq \theta_i \leq \Theta_{V,W}$ for all $i$ (in general, with strict inequalities).
	\SELF{If $V\neq\{0\}$. If $W = \{0\}$ both are $\frac\pi2$}
The product of $\sin\theta_i$'s has been studied in \cite{Afriat1957,Miao1992}, but not linked to a particular angle.

\begin{example}\label{ex:Upsilon via wedge}
	In \Cref{ex:contraction}, 
	$A\wedge B = -2u_{1235} + 2u_{1245} + 2u_{1234} + 2u_{1345}-u_{2345}$,
	$A\wedge C = 2u_{12345}$ and 
	$C\wedge D = 0$, 
	so that
		\OMIT{\ref{pr:Upsilon}\ref{it:exterior product}} 
	$\Upsilon_{[A],[B]} = \sin^{-1}\frac{\sqrt{17}}{6} \cong 43.4^\circ$,
	$\Upsilon_{[A],[C]} = \sin^{-1}\frac{\sqrt{2}}{3} \cong 28.1^\circ$ and
	$\Upsilon_{[C],[D]} = 0$.
\end{example}

\begin{proposition}\label{pr:formula complementary angle bases}
	Given bases $(v_1,\ldots,v_p)$ of $V$ and $(w_1,\ldots,w_q)$ of $W$, let $\AA_{p \times p} =\big(\inner{v_i,v_j}\big)$,
	$\BB_{q \times q} =\big(\inner{w_i,w_j}\big)$ and 
	$\CC_{q \times p} =\big(\inner{w_i,v_j}\big)$. Then
	$\sin^2 \Upsilon_{V,W} = \frac{\det(\AA - \bar{\CC}^T \BB^{-1} \CC)}{\det \AA} = \frac{\det(\BB - \CC \AA^{-1} \bar{\CC}^T)}{\det \BB}$.
\end{proposition}
\begin{proof}
	For $A=v_1\wedge\cdots\wedge v_p$ and $B=w_1\wedge\cdots\wedge w_q$,
	$\|A\wedge B\|^2 = \begin{vsmallmatrix}
		\AA & \bar{\CC}^T \\ 
		\CC & \BB
	\end{vsmallmatrix}$,
	so the result follows from Schur's identity 
		\OMIT{Schur: $\begin{vsmallmatrix}
				A & B \\ 
				C & D
			\end{vsmallmatrix} =$ \\ $\det(D)\det(A-BD^{-1}C)=$ \\ $\det(A)\det(D-CA^{-1}B)$}
	and \Cref{pr:Upsilon}\ref{it:exterior product}.
\end{proof}

\begin{example}
	In \Cref{ex:formula distinct dim} we find $\Upsilon_{V,W} = \Upsilon_{W,V} = 45^\circ$, 
	so lengths projected from $V$ to $W^\perp$ contract by $\frac{\sqrt{2}}{2}$, as do areas from $W$ to $V^\perp$.
	Principal angles confirm it: $V$ and $W^\perp$ have only $45^\circ$, while $W$ and $V^\perp$ have $0^\circ$ and $45^\circ$.
	In \Cref{ex:formula bases} we find
		\SELF{or \ref{pr:complementary orthon bases} with orthon bases $(\frac{v_1}{\sqrt{2}},\frac{v}{\sqrt{6}})$ and $(\frac{v_1}{\sqrt{2}},\frac{w}{\sqrt{2}})$}
	$\Upsilon_{V,W}=0$, as $V\cap W\neq\{0\}$.
		\OMIT{\ref{pr:Upsilon}\ref{it:Upsilon 0 pi2}}
\end{example}

\begin{corollary}\label{pr:complementary orthon bases}
	$\sin^2 \Upsilon_{V,W} = \det(\mathds{1}_{p\times p} - \bar{\mathbf{P}}^T \mathbf{P}) = \det(\mathds{1}_{q\times q} - \mathbf{P}\bar{\mathbf{P}}^T)$, 
	where $\mathbf{P}_{q\times p}$ is a matrix for the orthogonal projection $V\rightarrow W$ in orthonormal bases of $V$ and $W$.
\end{corollary}

If $V\pperp W$ and $V\cap W\neq\{0\}$, no more information can be extracted from $\Theta_{V,W}$ or $\Upsilon_{V,W}$.
But we can use a third angle, based on $\Theta_{V^\perp,W}$.

This angle is linked to a product of Grassmann-Cayley algebra \cite{Browne2012,Mandolesi_Contractions}.
A choice of a unit $\Omega \in \bigwedge^n X$ (an \emph{orientation} of $X$) induces an isometry (conjugate-linear, if $\F=\C$) $*:\bigwedge^p X \rightarrow \bigwedge^{n-p} X$ given by $A^* = A \lcontr \Omega$, with $[A^*] = [A]^\perp$ and $\|A^*\|=\|A\|$ (if $\F=\R$, $*$ is the usual Hodge star).
It gives a \emph{regressive product} $\vee$ defined by $(A\vee B)^* = A^* \wedge B^*$.
If $A\in\bigwedge^p X$ and $B\in\bigwedge^q X$ then $A\vee B \in\bigwedge^{p+q-n} X$ and $A \vee B = (-1)^{(n-p)(n-q)} B\vee A$.
For an orthonormal basis $(v_1,\ldots,v_n)$, orientation $v_{1\cdots n}$ and $\ii,\jj\in\II^n$,
if $\ii\cup\jj = 1\cdots n$ then 
$v_{\ii} \vee v_{\jj} = \epsilon_{\jj'\ii'} v_{\ii\cap\jj}$,
otherwise $v_{\ii} \vee v_{\jj} = 0$.
For blades, if $[A]+[B] = X$ then $[A\vee B] = [A]\cap [B]$, otherwise $A \vee B = 0$.

\begin{definition}\label{df:Psi}
	$\Psi_{V,W}=\frac{\pi}{2}-\Theta_{V^\perp,W}$ is the \emph{supplementation angle} of $V,W\in G(X)$.
\end{definition}

By \ref{it:Psi 0 pi2} below, $\Psi_{V,W} = 0$ unless $V$ and $W$ are \emph{supplementary} (we use the term for $V+W=X$, not $V\oplus W=X$), 
	\SELF{`complementary' means $V\oplus W=X$}
and is $\frac\pi2$ when each subspace contains the orthogonal complement of the other.

\begin{proposition}\label{pr:Psi}
	Let $V\in G_p(X)$ and $W\in G_q(X)$ be represented by blades $A$ and $B$, respectively, and $n = \dim X$.
	\begin{enumerate}[i)]
		\item $\Psi_{V,W} = 0 \Leftrightarrow V+W\neq X$, and $\Psi_{V,W} = \frac{\pi}{2} \Leftrightarrow V^\perp \subset W$.\label{it:Psi 0 pi2}\SELF{$\Leftrightarrow W^\perp \subset V$ \\ $V^\perp \pperp W$}
		\item $\sin \Psi_{V,W} = \frac{\|P_{W} (A^*)\|}{\|A\|}$. \label{it:Psi arcsin}
		\item $\Psi_{V,W} = \theta_{\bigwedge^{n-p} (V^\perp), (\bigwedge^{n-p} W)^\perp}$. \label{it:Psi angle exterior algebra}
		\item $\Psi_{V,W} = \Upsilon_{V^\perp,W^\perp}$. \label{it:Psi Upsilon}
		\item $\Psi_{V,W} = \Psi_{W,V}$. \label{it:Theta perp symmetric} 	
		\item $\|A\vee B\| = \|A\|\|B\|\sin\Psi_{[A],[B]}$. \label{it:regressive product}
		\item If $V,W\neq\{0\}$ have principal angles $\theta_1,\ldots,\theta_m$ for $m=\min\{p,q\}$, and $r = \dim(V \cap W)$, then \label{it:sin Psi}
		\begin{equation}\label{eq:Theta A perp sin}
			\sin\Psi_{V,W} = 
			\begin{cases}
				0  &\text{if } V+W\neq X, \\
				\prod_{i=r+1}^m \sin\theta_i &\text{if } V+W= X \text{ and } V,W\neq X, \\
				1  &\text{if } V=X \text{ or } W=X.
			\end{cases}
		\end{equation}
	\end{enumerate}
\end{proposition}
\begin{proof}
	\emph{(\ref{it:Psi 0 pi2}--\ref{it:Theta perp symmetric})} Immediate.
		\OMIT{\ref{it:Psi 0 pi2}: \ref{pr:Theta}\ref{it:Theta 0 pi2} \\
			\ref{it:Psi angle exterior algebra}: \ref{it:Psi Upsilon}, \ref{pr:Upsilon}\ref{it:Upsilon angle exterior algebra} \\
			\ref{it:Theta perp symmetric}: \ref{it:Psi Upsilon}, \ref{pr:Upsilon}\ref{it:symmetry complementary}
			}
	\emph{(\ref{it:regressive product})} By \Cref{pr:Upsilon}\ref{it:exterior product},
	$\|A\vee B\| = \| A^*\wedge B^* \| = \|A\|\|B\|\sin\Upsilon_{[A]^\perp,[B]^\perp}$. 
		\OMIT{\ref{it:Psi Upsilon}}
	\emph{(\ref{it:sin Psi})} The first and last cases follow from \ref{it:Psi 0 pi2}.
	The second one follows from \ref{it:Psi Upsilon} and \Cref{pr:Upsilon}\ref{it:prod sin}, as
	$V^\perp \cap W^\perp = \{0\}$ 
		\OMIT{no null principal angle}
	and $V^\perp,W^\perp \neq \{0\}$ have principal angles $\theta_{r+1},\ldots,\theta_m$, 
		\OMIT{$r<m$, as $V\subset W$ or vice-versa gives other cases}
	by \Cref{pr:thetas perps}.
\end{proof}

By \ref{it:Psi 0 pi2}, $\Psi_{V,W}$ does not satisfy a triangle inequality (\eg take two planes in $\R^3$ and their intersection).
By \ref{it:Psi arcsin}, $\sin \Psi_{V,W}$ (squared, if $\F=\C$) is the projection factor from $V^\perp$ to $W$.
As \ref{it:sin Psi} only has sines of nonzero $\theta_i$'s, $\Psi_{V,W}$ still gives information when $V\cap W \neq \{0\}$, unlike $\Upsilon_{V,W}$. 
Note that $\Psi_{V,W} \neq \Psi_{V,P_W(V)}$ (\eg take a line and a plane in $\R^3$).

\begin{corollary}
	For $V\in G_p(X)$ and $W\in G_q(X)$, $\Psi_{V,W} = 0$ if $p+q<n$,
	$\Psi_{V,W} = \Upsilon_{V,W}$ if $p+q=n$, 
	and
	$\Upsilon_{V,W} = 0$ if $p+q>n$.
\end{corollary}
\begin{proof}
	Let $p+q=n$.
	Then $V+W\neq X \Leftrightarrow V\cap W \neq \{0\}$, in which case both angles are $0$.
	If a subspace is $X$, the other is $\{0\}$, and both angles are $\frac\pi2$.
	Otherwise, we have the second case in \eqref{eq:Theta A perp sin} with $r=0$.
\end{proof}

\begin{corollary}
	$\max\{\Upsilon_{V,W},\Psi_{V,W}\} \leq \theta_i \leq \Theta_{V,W}$ for any nonzero principal angle $\theta_i$.
\end{corollary}

If there is no $\theta_i \neq 0$ we can have $\Theta_{V,W} < \Upsilon_{V,W}$ or $\Theta_{V,W} < \Psi_{V,W}$: \eg $\Theta_{\{0\},W} = \Theta_{V,X} = 0$ but $\Upsilon_{\{0\},W} = \Psi_{V,X} = \frac\pi2$.

\begin{example}
	In \Cref{ex:contraction}, the orientation $e_{1\cdots 5}$ gives
	$A\vee B = 0$,
	$A\vee C = 2$ and 
	$C\vee D = 2u_3$, so 
		\OMIT{\ref{pr:Psi}\ref{it:regressive product}}
	$\Psi_{[A],[B]} = 0$,
	$\Psi_{[A],[C]} = \sin^{-1}\frac{\sqrt{2}}{3} \cong 28.1^\circ$ and
	$\Psi_{[C],[D]} = \sin^{-1}\frac{\sqrt{2}}{2} = 45^\circ$.
	With \Cref{ex:Upsilon via wedge}, this shows all nonzero principal angles of $[A]$ and $[B]$ are in $[43.4^\circ, 80.4^\circ]$, those of $[A]$ and $[C]$ are in $[28.1^\circ,76.4^\circ]$, and those of $[C]$ and $[D]$ are in $[45^\circ,90^\circ]$.
\end{example}

\begin{proposition}
	Given a basis $(v_1,\ldots,v_p)$ of $V$ and orthonormal bases $(w_1,\ldots,w_q)$ of $W$ and $\beta = (u_1,\ldots,u_n)$ of $X$, with $p+q\geq n$, let 
	$\AA_{p \times p} =\big(\inner{v_i,v_j}\big)$
	and for $\ii = (i_1,\ldots,i_{n-p})\in \II^q_{n-p}$ let $(\mathbf{M}_\ii)_{n \times n} =\big(v_1 \,\cdots\, v_p \ w_{i_1} \,\cdots\, w_{i_{n-p}} \big)$, with the vectors decomposed in $\beta$ as columns. Then
	$\sin^2 \Psi_{V,W} = \frac{1}{\det \AA} \sum_{\ii \in \II^q_{n-p}} |\det\mathbf{M}_\ii|^2$.
\end{proposition}
\begin{proof}
	Let $A=v_1\wedge\cdots \wedge v_p$, and assume $\|A\|^2 = \det \AA = 1$ for simplicity.
	By \Cref{pr:Psi}\ref{it:Psi arcsin},
	$\sin^2 \Psi_{V,W} = \|P_{W} (A^*)\|^2 = \sum_{\ii\in\I_{n-p}^q} |\inner{w_\ii,A^*}|^2 = \sum\limits_{\ii\in\I_{n-p}^q} |\inner{A\lcontr u_{1\cdots n},w_\ii}|^2 = \sum\limits_{\ii\in\I_{n-p}^q} |\inner{u_{1\cdots n},A\wedge w_\ii}|^2 = \sum\limits_{\ii \in \II^q_{n-p}} |\det\mathbf{M}_\ii|^2$.
\end{proof}

\begin{example}
	In \Cref{ex:formula distinct dim}, $p+q<n$, so $\Psi_{V,W} = 0$.
	In \Cref{ex:formula bases}, 
	$\AA=\begin{psmallmatrix}
		2 & -1 \\ -1 & 2
	\end{psmallmatrix}$,
	$\mathbf{M}_1=\begin{psmallmatrix}
		1 & 0 & 1 \\ -\xi & \xi & 0 \\ 0 & -\xi^2 & 0
	\end{psmallmatrix}$
	and
	$\mathbf{M}_2=\begin{psmallmatrix}
		1 & 0 & 0 \\ -\xi & \xi & \xi \\ 0 & -\xi^2 & 0
	\end{psmallmatrix}$
	give $\Psi_{V,W} = \sin^{-1}\sqrt{\frac{2}{3}} \cong 54.7^\circ$ ($= \Theta_{V,W}$ as the only $\theta_i \neq 0$  has this value).
\end{example}

\begin{example}
	In $\R^4$, let $V=[v_{12}]$ and $W=[w_{12}]$ for $v_1=(1,-1,0,1)$, $v_2=(0,1,1,-1)$, $w_1=(1,0,0,0)$ and $w_2=(0,0,1,0)$.
	With
	$\mathbf{M}_{12}=\begin{psmallmatrix}
		1 & 0 & 1 & 0 \\ -1 & 1 & 0 & 0 \\ 0 & 1 & 0 & 1 \\ 1 & -1 & 0 & 0
	\end{psmallmatrix}$
	we obtain $\Psi_{V,W} = 0$, so $V+W \neq \R^4$.
\end{example}

\section{Other asymmetric metrics}\label{sc:Other asymmetric metrics}

With $G_p(\F^n) \subset G_p(\F^{n+1})$ 
induced by the canonical inclusion $\F^n \subset \F^{n+1}$, 
$G_p^\infty = \bigcup_{n\in\N} G_p(\F^n)$ is the \emph{infinite Grassmannian} of all $p$-subspaces in all $\F^n$'s, and $G^\infty = \bigcup_{n\in\N} G(\F^n) = \bigcup_{p\in\N} G_p^\infty$ is the \emph{infinite full Grassmannian} of all subspaces in all $\F^n$'s (`doubly infinite Grassmannian' in \cite{Ye2016}).

A method in \cite{Ye2016} extends metrics $d_p$ from $G_p^\infty$ to $G^\infty$:
for $V\in G_p^\infty$ and $W\in G_q^\infty$ with $p\leq q$,
it shows $\delta = \min\{d_p(V,U):U\in G_p(W)\} = \min\{d_q(W,Y):Y\in G_q^\infty, Y\supset V\}$,
	\SELF{as in \Cref{pr:extrema}, but we do not require $p\leq q$, and our sets mix different dimensions}
then turns it into a metric $d = \max\{d_q(W,Y):Y\in G_q^\infty, Y\supset V\}$
via an ad hoc inclusion of principal angles $\theta_{p+1} = \cdots = \theta_q = \frac\pi2$.
	\SELF{if a $q$-subspace can contain $V$ and $q-p$ extra dimensions in $W^\perp$, what explains the use of $G_q^\infty$}
One can skip $\delta$ and include these angles in $d_p$, as both have the same formula.
	\SELF{the minimal distances occur at $W_p$ and $V\oplus W_\perp$, which have with $V$ the same nonzero $\theta_i$'s as $W$} 
Most metrics obtained have a fixed value if $p\neq q$. 
From $d_{pF}$ the method gives $d_s$, and from $d_g$ and $d_{cF}$
	\CITE{their Grassmann and Procrustes metrics}
new metrics with a nonzero minimum value for $p\neq q$.

We will prove the following result, giving a simpler way to extend metrics on $G_p^\infty$ to asymmetric metrics on $G^\infty$ (which restrict to $G(X)$).

\begin{theorem}\label{pr:asym metrics}
	Let $d_p:G_p^\infty \times G_p^\infty \rightarrow [0,\infty]$ for $p = 0,1,2,\ldots$ be metrics such that, for $p \neq 0$:
	\begin{enumerate}[I.]
		\item $d_p(V,W) = f_p(\theta_1,\ldots,\theta_p)$ for a nondecreasing function $f_p$ of the principal angles $\theta_1,\ldots,\theta_p$ of $V,W\in G_p^\infty$. \label{it:dp fp}
		
		\item $f_q(0,\ldots,0,\theta_1,\ldots,\theta_p) = f_p(\theta_1,\ldots,\theta_p)$ for $q>p$ and any $\theta_1,\ldots,\theta_p$.\label{it:fr fp}
	\end{enumerate}	
	Also, let $\diam G_p^\infty = \sup\{d_p(V,W):V,W\in G_p^\infty\}$
		\SELF{can be $\infty$ for a metric non-equivalent to  \Cref{tab:metrics same dim}}
	and $\inf_p$ be the infimum taken in $[0,\diam G_p^\infty]$.
	An asymmetric metric $d:G^\infty \times G^\infty \rightarrow [0,\infty]$ is given by $d(V,W) = \inf_p\{d_p(V,W'):W'\in G_p(W)\}$ with $p=\dim V$.
\end{theorem}

Note that $d(V,W) \leq \diam G_p^\infty$ for $p=\dim V$.

Use of $\inf_p$ instead of $\min$ is crucial, 
as $G_p(W) = \emptyset$ for $p>\dim W$ and $\min \emptyset$ is not defined.
On the other hand, the infimum $\inf_S$ taken in an ordered set $S$ with greatest element $M$ satisfies $\inf_S \emptyset = M$ (by its definition as the greatest lower bound: any $s\in S$ is a lower bound of $\emptyset$, 
since $\emptyset$ has no element smaller than $s$ \cite[p.\,261]{Bajnok2020}).
This may seem like a technicality, but is an important property of the infimum,
and plays a central role in our method.

To prove the Theorem we will need some results.
The following are particular cases of \cite[Cor.\,3.1.3]{Horn1991}, which we state for convenience.

\begin{proposition}\label{pr:singular values}
	Let $\sigma_1 \geq \cdots \geq \sigma_p$ be the singular values of a $q\times p$ matrix $\PP$, with $p\leq q$.
		\SELF{pode tirar $p\leq q$, $k\leq q-p$ se definir $\sigma_i,\sigma_i'=0$ para $i$ além do max. No \cite{Horn1991} podia deletar rows e columns ao mesmo tempo}
	\begin{enumerate}[i)]
		\item If $\sigma_1' \geq \cdots \geq \sigma_p'$ are the singular values of a $(q-k) \times p$ matrix formed by deleting $k\leq q-p$ rows of $\PP$ then $\sigma_i' \leq \sigma_i$ for $1\leq i \leq p$.
		
		\item If $\sigma_1' \geq \cdots \geq \sigma_{p-k}'$ are the singular values of a $q \times (p-k)$ matrix formed by deleting $k$ columns of $\PP$ then $\sigma_i' \geq \sigma_{i+k}$ for $1\leq i \leq p-k$.\SELF{$\sigma_i \geq \sigma_i'$ also}
	\end{enumerate}
\end{proposition}

\begin{corollary}\label{pr:theta' theta}
	Let $\theta_1\leq\cdots\leq\theta_p$ be the principal angles of $V\in G_p(X)$ and $W\in G_q(X)$, with $p\leq q$.
	\begin{enumerate}[i)]
		\item If $\theta_1'\leq\cdots\leq\theta_p'$ are the principal angles of $V$ and $W' \in G_p(W)$ then $\theta_i' \geq \theta_{i}$ for $1\leq i \leq p$. \label{it:theta V theta W'}
		\item If $\theta_1'\leq\cdots\leq\theta_r'$ are the principal angles of $V' \in G_r(V)$ and $W$ then $\theta_i' \leq \theta_{i+p-r}$ for $1\leq i \leq r$.\label{it:theta V' theta W}\SELF{$\theta_i \leq \theta_i'$ also}
	\end{enumerate}
\end{corollary}
\begin{proof}
	Follows by extending orthonormal
		\SELF{só precisa pros $W$'s}
	bases of $V'$ and $W'$ to $V$ and $W$, and applying \Cref{pr:singular values} to matrices representing, in these bases, orthogonal projections $V\rightarrow W$, $V \rightarrow W'$ and $V' \rightarrow W$.
\end{proof}

\begin{proposition}\label{pr:d cases}
	Let $V\in G_p^\infty$, $W\in G_q^\infty$, and if $p,q\neq 0$ let $\theta_1,\ldots,\theta_p$ be the principal angles of $V$ and $W$, and $W_P$ be a projective subspace of $W$ \wrt $V$. With the notation of \Cref{pr:asym metrics},
		\OMIT{\Cref{df:PO}}
		\SELF{$d(V,W) \leq \diam G_p^\infty$ $\forall p,q$}
	\begin{equation}\label{eq:d cases}
		d(V,W) = \begin{cases}
			d_p(V,W_P) = f_p(\theta_1,\ldots,\theta_p) &\text{if } 0<p\leq q, \\
			\diam G_p^\infty &\text{otherwise}.
		\end{cases}
	\end{equation}
\end{proposition}
\begin{proof}
	For $p=0$, $d(V,W) = 0 = \diam G^\infty_0$.
	For $p>q$, $G_p(W) = \emptyset$, so $d(V,W)  = \inf_p \emptyset = \diam G_p^\infty$.
	For $0<p\leq q$, Propositions \ref{pr:thetas WP Wperp} and \ref{pr:theta' theta}\ref{it:theta V theta W'} show
	the principal angles of $V$ and $W_P$ are the $\theta_i$'s, which bound from below those of $V$ and $W'\in G_p(W)$,
	so $d_p(V,W_P) \leq d_p(V,W')$ by \ref{it:dp fp}.
\end{proof}

\begin{lemma}\label{pr:d subspace smaller}
	Given $U,V\in G_p^\infty$ and $U' \in G_r(U)$, there is $V' \in G_r(V)$ with $d_r(U',V') \leq d_p(U,V)$.
\end{lemma}
\begin{proof}
	Let $U$ and $V$ have principal angles $\theta_1,\ldots,\theta_p$, and $U'$ and $V$ have $\theta_1',\ldots,\theta_r'$.
	With \ref{it:dp fp} and \ref{it:fr fp},
	Propositions \ref{pr:thetas WP Wperp} and \ref{pr:theta' theta}\ref{it:theta V' theta W} give $d_p(U,V) = f_p(\theta_1,\ldots,\theta_p) \geq f_p(0,\ldots,0,\theta_1',\ldots,\theta_r') = f_r(\theta_1',\ldots,\theta_r') = d_r(U',V_P)$, for a projective subspace $V_P$ of $V$ \wrt $U'$.
\end{proof}

\begin{lemma}\label{pr:diameter nondecr}
	$\diam G_p^\infty$ is non-decreasing on $p$.
\end{lemma}
\begin{proof}
	Any $G_p^\infty$ has orthogonal subspaces, so, for $0<p<q$, $\diam G_p^\infty = f_p(\frac\pi2,\ldots,\frac\pi2) = f_q(0,\ldots,0,\frac\pi2,\ldots,\frac\pi2) \leq f_q(\frac\pi2,\ldots,\frac\pi2) = \diam G_q^\infty$.
		\OMIT{\ref{it:dp fp}, \ref{it:fr fp}}
\end{proof}

The reason why we use $G_p^\infty$ is that $\diam G_p(\F^n)$ decreases for large $p$.
	\SELF{depending on $n$ and the metric. For $p=n$ it is $0$}
We can now prove the theorem:

\begin{proof}[Proof of \Cref{pr:asym metrics}]
	Let $U\in G_r^\infty$, $V\in G_p^\infty$, $W\in G_q^\infty$.
	By \eqref{eq:d cases}, if $p \neq 0$ and $d(V,W)=0$ then $d_p(V,W_P)=0$,
		\OMIT{$\diam G_p \neq 0$ for $p\neq 0$}
	so $V=W_P \subset W$.
	This gives \Cref{df:asymmetric metric}\ref{it:d=0 x=y}, and we prove $d(U,W) \leq d(U,V) + d(V,W)$. 
	If $r>p$, $d(U,W) \leq \diam G_r^\infty = d(U,V)$.
		\OMIT{\eqref{eq:d cases}}
	If $r\leq p$ and $p>q$, \Cref{pr:diameter nondecr} gives 
	$d(U,W) \leq \diam G_r^\infty \leq \diam G_p^\infty = d(V,W)$.
		\OMIT{\eqref{eq:d cases}}
	If $0 < r\leq p \leq q$, 
		\OMIT{$r=0$ is trivial}
	Propositions \ref{pr:d cases} and \ref{pr:d subspace smaller} give $V_P\in G_r(V)$, $W_P\in G_p(W)$ and $W'\in G_r(W_P)$ with $d(U,V) = d_r(U,V_P)$
	and $d(V,W) = d_p(V,W_P) \geq d_r(V_P,W')$,
	so $d(U,W) \leq d_r(U,W') \leq d_r(U,V_P) + d_r(V_P,W') \leq d(U,V) + d(V,W)$.
		\OMIT{$d=\inf d_r$ \\ $d_r$ metric \\ previous line}
\end{proof}

\begin{table}[]
	\centering
	\renewcommand{\arraystretch}{1}
	\begin{tabular}{ccc}
		\toprule
		Metric on $G_p^\infty$ & $\diam G_p^\infty$ & Asymmetric metric on $G^\infty$
		\\
		\cmidrule(lr){1-1} \cmidrule(lr){2-2}  \cmidrule(lr){3-3} 
		$d_g$ & $\frac\pi2 \sqrt{p}$ & $\sqrt{\sum_{i=1}^p \theta_i^2}$  if $p\leq q$, otherwise $\frac\pi2 \sqrt{p}$
		\\[3pt]
		$d_{cF}$ & $\sqrt{2p}$ & $2\sqrt{\sum_{i=1}^p \sin^2 \frac{\theta_i}{2}}$ if $p\leq q$, otherwise $\sqrt{2p}$
		\\[3pt] 
		$d_{pF}$ & $\sqrt{p}$ & $\sqrt{\sum_{i=1}^p \sin^2 \theta_i}$ if $p\leq q$,  otherwise $\sqrt{p}$
		\\[3pt]
		$d_{FS}$ & $\frac\pi2$ & $\cos^{-1}(\prod_{i=1}^p \cos\theta_i)$ if $p\leq q$, otherwise $\frac\pi2$
		\\[3pt]
		$d_{c\wedge}$ & $\sqrt{2}$ & $\sqrt{2-2\prod_{i=1}^p \cos\theta_i}$ if $p\leq q$, otherwise $\sqrt{2}$
		\\[3pt] 
		$d_{BC}$ & $1$ & $\sqrt{1-\prod_{i=1}^p \cos^2\theta_i}$ if $p\leq q$, otherwise $1$
		\\[3pt] 
		$d_{A}$ & $\frac\pi2$ & $\theta_p$ if $p\leq q$, otherwise $\frac\pi2$
		\\[3pt] 
		$d_{c2}$ & $\sqrt{2}$ & $2\sin\frac{\theta_p}{2}$ if $p\leq q$, otherwise $\sqrt{2}$
		\\[3pt]
		$d_{p2}$ & $1$ & $\sin\theta_p$ if $p\leq q$, otherwise $1$
		\\
		\bottomrule
	\end{tabular}
	\caption{Diameter of $G_p^\infty$ and asymmetric distances from $V\in G_p^\infty$  to $W\in G_q^\infty$ ($p\neq 0$)}
	\label{tab:asymmetric metrics}
\end{table}

\Cref{tab:asymmetric metrics} has the asymmetric metrics obtained from metrics of \Cref{tab:metrics same dim} ($d_{FS}$, $d_{c\wedge}$, $d_{BC}$ and $d_{p2}$ give $\Theta_{V,W}$, $2\sin\frac{\Theta_{V,W}}{2}$, $\sin \Theta_{V,W}$ and the containment gap). 
They are not trivial for $p\neq q$, and have a minimum of $0$ when $V\subset W$, 
avoiding the problems seen in \Cref{sc:Distances on the full Grassmannian}.
Inequalities of \Cref{sc:Distance inequalities} still hold, except that some become equalities if $p>q$.
Hence these asymmetric metrics are topologically equivalent, inducing the same backward, forward and symmetric topologies described in \Cref{sc:metric}, what suggests these are natural topologies for the full Grassmannian.

If necessary, asymmetric metrics can be symmetrized, but results leave something to be desired.
For example, $\hat{\Theta}_{V,W} = \max\{\Theta_{V,W},\Theta_{W,V}\}$ gives the full Fubini-Study metric, which is trivial for different dimensions, and
$\check{\Theta}_{V,W} = \min\{\Theta_{V,W},\Theta_{W,V}\}$ is a common angle \cite{Gluck1967,Gunawan2005,Jiang1996} that projects the smaller subspace on the larger one, but does not satisfy a triangle inequality.
\OMIT{two lines and their plane}
These angles are linked to the scalar and Hestenes products of Clifford algebra \cite{Mandolesi_Products},
and are the $d^\phi$ and $\delta^\phi$ obtained from $d_{FS}$ in \cite{Ye2016}\footnote{There is a small error in \cite[p.\,1189]{Ye2016}: $c_\phi$ should be $\frac{\pi}{2}$, not $1$.}.
And $\bar{\Theta}_{V,W} = \frac{\Theta_{V,W}+\Theta_{W,V}}{2}$ is a nontrivial metric, but has a minimum of $\frac\pi4$ for different dimensions, and does not seem to have nice properties.

\section{Conclusion}\label{sc:conclusion}

The main Grassmannian metrics have been extended to asymmetric metrics which induce natural topologies on the full Grassmannian of subspaces of different dimensions.
The Fubini-Study metric extends to an asymmetric angle which we studied in detail, obtaining many properties that facilitate its use and computation.
It remains to be seen whether the other asymmetric metrics also have nice properties, and how they all fare in applications.

An aspect of the Fubini-Study distance is that it quickly approaches its maximum value of $\frac\pi2$ if various principal angles are large, or even if a large number of them are small but nonzero. 
This is relevant for quantum entanglement and decoherence, but might be inconvenient for other applications,
so this distance may perhaps be more appropriate for problems involving a moderate number of small perturbations.

Other results for asymmetric angles can be found in \cite{Mandolesi_Grassmann,Mandolesi_Products}.
They can also be computed via Clifford geometric product \cite{Dorst2007,Mandolesi_Products}: if unit blades $A$ and $B$ represent $V\in G_p(X)$ and $W \in G_q(X)$ then $\cos \Theta_{V,W}$ and $\sin \Upsilon_{V,W}$ are, respectively, the norms of the components of grades $q-p$ and $p+q$ in $AB$.
For applications using oriented subspaces, \cite{Mandolesi_Products} has a variant of $\Theta_{V,W}$ that encodes the relative orientation of subspaces.

\appendix

\section{Some inequalities}\label{sc:Distance inequalities}

First we prove the spherical triangle inequality:

\begin{proof}[Proof of \Cref{pr:spherical triangle inequality lines}.]
	Assume distinct lines and $\theta_{K,L} \neq \frac\pi2$.
	For a unit $w\in L$, let $v = \frac{P_K w}{\|P_K w\|}$, and if $\theta_{J,L} \neq \frac\pi2$ let $u = \frac{P_J w}{\|P_J w\|}$, otherwise take any unit $u\in J$,
	so $\theta_{J,L}=\theta_{u,w}$ and $\theta_{K,L}=\theta_{v,w}$.
		\OMIT{$\inner{u,w}\geq 0$ and $\inner{v,w} > 0$}
	Let 
	$u^\perp = \frac{u-P_L u}{\|u-P_L u\|} \in L^\perp$ and 
	$v^\perp = \frac{v-P_L v}{\|v-P_L v\|} \in L^\perp$.
	As $P_L u=w \cos\theta_{u,w}$ and $\|u-P_L u\|=\sin\theta_{u,w}$,
	\OMIT{$\inner{u-P_L u,u-P_L u}$ \\ $= 1+\|P_L u\|^2-2Re\inner{u,P_L u}$ \\ $=1-\|P_L u\|^2 = 1-\cos^2\theta$}
	we find $u = w\cos\theta_{u,w} + u^\perp\sin\theta_{u,w}$,
	and likewise 
	$v = w\cos\theta_{v,w} + v^\perp\sin\theta_{v,w}$.
	Thus
	$\cos\theta_{J,K} = |\inner{u,v}| 
	= |\cos\theta_{u,w} \cos\theta_{v,w} 
	+ \inner{u^\perp,v^\perp} \sin\theta_{u,w} \sin\theta_{v,w}| 
	\leq \cos(\theta_{u,w}-\theta_{v,w})$,
	so that
	$\theta_{J,K} \geq \theta_{u,w}-\theta_{v,w}$.
	\OMIT{$\theta_{J,K} \in [0,\frac\pi2]$ \\ $\theta_{u,w}-\theta_{v,w} \in [-\frac\pi2,\frac\pi2]$}
	
	Equality gives $\theta_{u,w} = \theta_{J,K} + \theta_{v,w} > \theta_{v,w}$ and $\inner{u^\perp,v^\perp} = 1$, so $v^\perp = u^\perp$ and
	$v = w \cos\theta_{v,w} + \frac{u-P_L u}{\|u-P_L u\|} \sin\theta_{v,w}
	= u   \frac{\sin\theta_{v,w}}{\sin \theta_{u,w}} + w   \frac{\sin(\theta_{u,w}-\theta_{v,w})}{\sin \theta_{u,w}}
	= \kappa u+\lambda w$ with $\kappa,\lambda > 0$.
	Conversely, if $v=\kappa u + \lambda w$ with $\kappa,\lambda \geq 0$ and $\inner{u,w} \geq 0$, we find $\theta_{u,w} = \theta_{u,v} + \theta_{v,w}$.
	\OMIT{Dá trabalho provar a partir da definição, mas é padrão}
	As $\inner{u,v} = \kappa \|u\|^2 + \lambda \inner{u,w} \geq 0$, and likewise $\inner{v,w} \geq 0$, we have $\theta_{J,K}=\theta_{u,v}$, $\theta_{J,L}=\theta_{u,w}$ and $\theta_{K,L}=\theta_{v,w}$.
\end{proof}

The metrics in \Cref{tab:metrics same dim} are often said to be topologically equivalent, but we could not locate a proof.
Also, an important reference on the subject \cite[p.\,338]{Edelman1999} gives, for $V\neq W$, strict inequalities $d_g>d_{FS}$, $d_{cF}>d_{c2}$ and $d_{pF}>d_{p2}$, what is incorrect (\eg take $\theta_1=\cdots=\theta_{p-1}=0$ and $\theta_p\neq 0$).
The proofs below set the record straight.

\begin{proposition}
	For distinct $V,W\in G_p(X)$:
	\begin{enumerate}[i)]
		\item $\frac\pi2 d_{pF} \geq d_g > d_{cF} > d_{pF}$.
		\item $\frac\pi2 d_{BC} \geq d_{FS} > d_{c\wedge} > d_{BC}$.
		\item $\frac\pi2 d_{p2} \geq d_{A} > d_{c2} > d_{p2}$.
	\end{enumerate}
\end{proposition}
\begin{proof}
	Follows from the formulas in \Cref{tab:metrics same dim}, as for distinct lines $K$ and $L$ we have $\frac\pi2 g_{K,L} \geq \theta_{K,L} > c_{K,L} > g_{K,L}$.
\end{proof}

\begin{proposition}
	Let $V,W\in G_p(X)$. If $\dim(V\cap W) < p-1$ then
		\SELF{no lugar de $\sqrt{p}$ pode ser $\sqrt{p-r}$, com $r=\dim V\cap W$}
	\begin{enumerate}[i)]
		\item $\sqrt{p}\, d_{A} \geq d_g > d_{FS} > d_{A}$. \label{it:angular ineq}
		\item $\sqrt{p}\, d_{c2} \geq d_{cF} > d_{c\wedge} > d_{c2}$. \label{it:chordal ineq}
		\item $\sqrt{p}\, d_{p2} \geq d_{pF} > d_{BC} > d_{p2}$. \label{it:gap ineq}
	\end{enumerate}
	If $\dim(V\cap W)\geq p-1$ the strict $>$'s become equalities.
\end{proposition}
\begin{proof}
	If $\dim(V\cap W)\geq p-1$ then $\theta_i=0$ for $i \neq p$, and the distance formulas give the equalities.
	For $\dim(V\cap W) < p-1$ we prove only the second inequality in each item, as the others are simple.

	\emph{(\ref{it:angular ineq})}
	We show $\cos^{-1}(\cos\theta_1 \cdots \cos\theta_p) \leq \sqrt{\theta_1^2 + \cdots + \theta_p^2}$ 
	for $\theta_1,\ldots,\theta_p \in [0,\frac\pi2]$, 
	with strict inequality if $\theta_{p-1},\theta_p\neq 0$.
	For $p=2$ this is done showing, for $f(x,y)=\cos^{-1}(\cos x \cos y)$, $g(x,y) = \sqrt{x^2+y^2}$ and $x,y\in\ ]0,\frac\pi2[$, that
	$\frac{\partial g}{\partial x} = \frac{x}{\sqrt{x^2+y^2}}$ is increasing on $x$, 
		\OMIT{$\frac{\partial^2 g}{\partial x^2}>0$}
	so $\frac{\partial g}{\partial x} > \frac{\sin x}{\sqrt{\sin^2 x + \tan^2 y}} = \frac{\partial f}{\partial x}$.
		\OMIT{$f(0,y) = g(0,y)$ \\ $y=\frac\pi2$ direto nas fórmulas}
	Assuming the result for some $p\geq 2$, 
	let $\theta_1,\ldots,\theta_{p+1} \in [0,\frac\pi2]$ 
	and $x = \cos^{-1}(\cos\theta_1 \cdots \cos\theta_p) \leq \sqrt{\theta_1^2 + \cdots + \theta_p^2}$.
	So
	$\cos^{-1}(\cos\theta_1 \cdots \cos\theta_{p+1}) = \cos^{-1}(\cos x \cos\theta_{p+1}) \leq \sqrt{x^2 + \theta_{p+1}^2} \leq \sqrt{\theta_1^2 + \cdots + \theta_{p+1}^2}$,
	\OMIT{$x \leq \sqrt{\theta_1^2 + \cdots + \theta_p^2}$ as only $\theta_p\neq 0$}
	and the first inequality is strict if $\theta_p,\theta_{p+1}\neq 0$ (so $x\neq 0$).
	
	\emph{(\ref{it:chordal ineq})}
	$d_{c\wedge} = \sqrt{2-2\prod_{i=1}^p \cos\theta_i}$ and $d_{cF} = \sqrt{2p-2\sum_{i=1}^p \cos\theta_i}$, 
	so we show $1-\prod_{i=1}^p x_i \leq p-\sum_{i=1}^p x_i$ for $x_1,\ldots,x_p \in [0,1]$, with strict inequality if $x_{p-1},x_p \neq 1$.
		\OMIT{$x_i=\cos\theta_i$}
	For $p=2$, $1-x_1x_2 \leq 1 - x_1x_2 + (1-x_1)(1-x_2) = 2-x_1-x_2$, with strict inequality if $x_1,x_2 \neq 1$.
	Assuming the result for some $p\geq 2$, let $x_1,\ldots,x_{p+1} \in [0,1]$ and $x = \prod_{i=1}^p x_i \geq 1-p + \sum_{i=1}^p x_i$.
	Then $1-\prod_{i=1}^{p+1} x_i = 1 - x x_{p+1} \leq 2 - x - x_{p+1} \leq p+1-\sum_{i=1}^{p+1} x_i$, and the first inequality is strict if $x_p,x_{p+1} \neq 1$.
	
	\emph{(\ref{it:gap ineq})} 
	$d_{BC} = \sqrt{1-\prod_{i=1}^p \cos^2\theta_i}$ and $d_{pF} = \sqrt{p-\sum_{i=1}^p \cos^2\theta_i}$, so the result follows as in \ref{it:chordal ineq}.
		\OMIT{$x_i=\cos^2\theta_i$}
\end{proof}

\providecommand{\bysame}{\leavevmode\hbox to3em{\hrulefill}\thinspace}
\providecommand{\MR}{\relax\ifhmode\unskip\space\fi MR }
\providecommand{\MRhref}[2]{%
	\href{http://www.ams.org/mathscinet-getitem?mr=#1}{#2}
}
\providecommand{\href}[2]{#2}

\end{document}